\documentclass[11pt,reqno]{amsart}

\usepackage{graphicx}
\usepackage{epstopdf}
\usepackage{psfrag}

\usepackage{amsfonts,amsmath,amssymb, amsthm}
\usepackage[usenames]{color}
\newtheorem{theorem}{Theorem}[section]

\newtheorem{lemma}[theorem]{Lemma}
\newtheorem{corollary}[theorem]{Corollary}
\newtheorem{remark}[theorem]{Remark}
\newtheorem{definition}[theorem]{Definition}

\newcommand{\be}{\begin{equation}}
\newcommand{\ee}{\end{equation}}
\newcommand{\bigslant}[2]{{\raisebox{.2em}{$#1$}\left/\raisebox{-.2em}{$#2$}\right.}}

\numberwithin{equation}{section}

\newcommand{\Ch}{{\rm Ch}}
\newcommand{\sfd}{{\sf d}}

\DeclareMathOperator{\diam}{diam}

\DeclareMathOperator{\supp}{spt}

\newcommand{\RCD}{\mathsf{RCD}}
\newcommand{\CD}{\mathsf{CD}}
\newcommand{\mm}{\mathfrak m}
\newcommand{\gopt}{{\rm{OptGeo}}}

\newcommand{\ppi}{{\mbox{\boldmath$\pi$}}}
\newcommand{\ggamma}{{\mbox{\boldmath$\gamma$}}}
\newcommand{\sggamma}{{\mbox{\scriptsize\boldmath$\gamma$}}}
\newcommand{\geo}{{\rm{Geo}}}
\newcommand{\e}{{\rm{e}}}
\newcommand{\LL}{\mathcal{L}}

\newcommand{\R}{\mathbb{R}}

\renewcommand{\d}{{\textrm {d}}}

\newcommand\Z{{\mathbb Z}}
\newcommand\N{{\mathbb N}}

\begin{document}

\title{On the universal cover and the fundamental group of an $\RCD^*(K,N)$-space}

\author{Andrea Mondino}  \thanks{A. Mondino,  MSRI-Berkeley\&Universit\"at Z\"urich, Institut f\"ur Mathematik. email: \textsf{andrea.mondino@math.uzh.ch}}
\author {Guofang Wei} \thanks{G. Wei,  Department of mathematics, University of California Santa Barbara.   email: \textsf{wei@math.ucsb.edu} Partially supported by NSF DMS 1506393}

\maketitle

\begin{abstract} 
The main goal of the paper is to prove the existence of the universal cover for  $\RCD^*(K,N)$-spaces. This generalizes earlier work of \cite{SW1, SW2} on the existence of universal covers for Ricci limit spaces. As a result, we also obtain several structure results on the (revised)  fundamental group of $\RCD^*(K,N)$-spaces. These are the first topological results for $\RCD^{*}(K,N)$-spaces without extra structural-topological assumptions (such as semi-local simple connectedness). 
\end{abstract}

\tableofcontents

\section{Introduction}

One of the most classical and  fundamental problems in geometry is to study the relation between curvature and topology. An excellent example of such an interplay is the celebrated Gauss-Bonnet Theorem which relates the total  integral of the  Gauss curvature with the Euler characteristic of a closed surface.

One of the simplest topological invariants of a topological space is the fundamental group, which encodes the information about non-contractible closed loops.  While the sectional curvature controls more than the fundamental group (in fact, as proved by Gromov \cite{Gromov-Betti}, sectional curvature bounds give information on  all Betti numbers)  and scalar curvature does not give any control on the fundamental group when the  dimension is at least 4 (see \cite{Carr}),    Ricci curvature controls the fundamental group very well, see e.g. \cite{Myers, Mil, ChGr, K-Wilking, An}, see also \cite{Wei, Wilking}.  A key technical point in all of these papers is that,  in order  to get information on the fundamental group, one works on the universal cover of the manifold; more generally, one could say that  the principle behind such works is that \emph{understanding the fundamental group of a manifold is equivalent to understand the geometry of its  universal cover}.

For topological manifolds, the existence and uniqueness of a universal cover are well known facts. Moreover,  for Riemannian manifolds, one can endow the universal cover with a Riemannian metric  so that the covering map is a local isometry; in this way,  the curvature assumptions on the base are inherited automatically by the cover.  
\\Extending such fundamental facts to more general spaces presents two problems:
\begin{enumerate}
\item   A priori the universal cover may not exist (see Remark \ref{rem:ExistUC} for more details).
\item  The curvature condition may not be local (see the discussion below). 
\end{enumerate}
In \cite{SW1, SW2}, C. Sormani and the second named author developed a general strategy  for showing the  existence of  a universal cover for complete proper length spaces satisfying appropriate regularity and geometric assumptions. One of their main achievements was to prove the existence of a universal cover  for  the Gromov-Hausdorff limits of manifolds  with Ricci curvature uniformly bounded from below (see also \cite{B-Plaut} for  a more general notion of covering space, called  uniform universal cover).
The goal of the present paper is to extend such results to the so called $\RCD^{*}(K,N)$-spaces; in the next paragraph we recall some basics facts about such spaces. 

\subsection{Lower Ricci curvature bounds for metric measure spaces}
While lower sectional curvature bounds are a property of the distance and make perfect sense in the framework of metric spaces, the natural setting being provided by the celebrated Alexandrov geometry, Ricci curvature lower bounds involve the \emph{interplay of distance and volume}. The natural framework is indeed given by metric measure spaces $(X,\sfd,\mm)$, where $(X,\sfd)$ is a complete and separable metric space and $\mm$ is a reference measure that plays the role of the volume measure in a non-smooth setting.

By using tools of optimal transportation, Lott-Villani \cite{Lott-Villani09} and Sturm \cite{Sturm06I}-\cite{Sturm06II}, introduced  the so called $\CD(K,N)$-condition  which corresponds to the non-smooth analog for a metric measure space to have Ricci curvature bounded below by $K$ and dimension bounded above by $N$.  Two fundamental properties of such condition are the compatibility with the smooth counterpart (i.e. a smooth Riemannian manifold satisfies  the  $\CD(K,N)$-condition if and only if it has actually Ricci $\geq K$ and dimension $\leq N$) and the stability with respect to  measured Gromov Hausdorff convergence (to this purpose see also \cite{GigliMondinoSavare}).  Despite remarkable partial results (see for instance \cite{Cavalletti12},\cite{CM2}, \cite{CavaSturm},  \cite{Rajala}) it is still unclear if the $\CD(K,N)$-condition satisfies a local-to-global property under mild regularity assumptions (e.g. essential non-branching) when $K \not= 0$ and $N<\infty$. 
\\To this aim, Bacher and Sturm \cite{BS2010} introduced an apriori weaker curvature condition called \emph{reduced curvature condition} and denoted by $\CD^{*}(K,N)$ which instead satisfies the local-to-global property  as well as the stability and  most of the geometric comparison consequences (e.g. sharp Bishop-Gromov \cite{CavaSturm}, sharp L\'evy-Gromov \cite{CM1},  sharp Brunn-Minkoski \cite{CM2}). Let us mention that   $\CD(K,N)$   imples $\CD^{*}(K,N)$   and the two are the  same  when $K =0$.  

A second issue regarding the $\CD(K,N) \backslash \CD^{*}(K,N)$-conditions above is that while on the one hand they allow for  Finsler structures, on the other hand it was understood since the work of Cheeger-Colding 
\cite{Cheeger-Colding97I},\cite{Cheeger-Colding97II},\cite{Cheeger-Colding97III}  that purely Finsler structures never appear as Ricci limit spaces.   Inspired by this fact, in \cite{Ambrosio-Gigli-Savare11b}, Ambrosio-Gigli-Savar\'e proposed a  strengthening of the $\CD$ condition in order to enforce, in some weak sense, a Riemannian-like behavior of spaces with a curvature-dimension bound (see also \cite{AmbrosioGigliMondinoRajala}); the finite dimensional refinement led to the so called $\RCD^{*}(K,N)$ condition. 
\\Such a strengthening consists in requiring that the space $(X,\sfd,\mm)$ is such that the Sobolev space $W^{1,2}(X,\sfd,\mm)$ is Hilbert, a condition we shall refer to as `infinitesimal Hilbertianity'. 
Since the $\RCD^{*}(K,N)$ condition is stable under convergence (see  \cite{Ambrosio-Gigli-Savare11b}, \cite{GigliMondinoSavare} and  \cite{EKS}),  the class of $\RCD^{*}(K,N)$-spaces includes Ricci limit spaces (no matter if collapsed or not). Moreover it contains weighted manifolds satisfying Bakry-\'Emery lower curvature bounds as well as their non-smooth limits, cones, warped products, etc..  For more details about $\RCD^{*}(K,N)$-spaces  we refer to Section~\ref{RCD_intro}.
\\For the moment let us summarize by saying that one can think of the $\RCD^{*}(K,N)$-spaces as the Ricci curvature analog of the celebrated Alexandrov spaces.

\subsection{Main results}
Since, as discussed above, the $\RCD^{*}(K,N)$-condition enjoys the local-to-global property, the universal cover (if we are able to  show its existence) will have good chances to be $\RCD^{*}(K,N)$ as well, and the above program consisting  in using the geometric information on the universal cover in order to infer structural properties of the fundamental group will have good chances too. This was indeed one of the main motivations for us to write the present paper, whose main result is the following:

\begin{theorem}\label{thm:UnCovRCD-noncompactIntro}
Let $(X,\sfd,  \mm)$ be an  $\RCD^*(K,N)$-space for some $K\in \R$, $N\in(1, +\infty)$.  Then $(X,\sfd,  \mm)$ admits a universal cover $(\tilde{X},\tilde{\sfd},\tilde{\mm})$ which is itself an  $\RCD^*(K,N)$-space. 
\end{theorem}

To prove the result we use the  criterions for the existence of universal covers for metric spaces established in  \cite{SW1, SW2}, see Theorem~\ref{thm:SWUniCov} and Theorem \ref{univstab} below.  To achieve such criterions  we will  take advantage of  the recent excess estimate and structure results established for $\RCD^*(K,N)$-spaces in \cite{GigliMosconi, MNRect}, see Theorem~\ref{AG-ineq} and Theorem \ref{RCD-reg} below. Let us stress that the proof presented here for $\RCD^{*}(K,N)$-spaces is more streamlined than the one in \cite{SW1, SW2} for Ricci limit spaces, as  we can work  intrinsically  on the space without going back and forth to the approximating smooth sequence (for more details about the simplifications see Remark \ref{rem:simplificationsCompact} and Remark \ref{rem:SimplicationNonCompact}). 
\\Let us now discuss some applications of Theorem \ref{thm:UnCovRCD-noncompactIntro}. Before stating them, recall that the \emph{revised fundamental group} is by definition the group of deck transformations of the universal cover, and is a quotient of the standard fundamental group (it coincides with the standard fundamental group if the universal cover is simply connected, a fact which is not clear already for the Ricci-limit spaces). A first application is a rather direct consequence of the compactness of $\RCD^{*}(K,N)$-spaces for $K>0$:

 \begin{theorem}[Theorem \ref{thm:finiteRev}]\label{thm:finiteRevIntro}
Let $(X,\sfd,\mm)$ be an $\RCD^*(K,N)$-space for some $K>0$ and $N\in (1,\infty)$. Then the revised fundamental group $\bar{\pi}_1(X)$ is finite.
\end{theorem}
Let us mention that the above result  was proved in \cite{BS2010} for non-branching $\CD^*(K,N)$-spaces  under the extra assumption of semi-local simple connectedness; recall that the extra assumption of semi-local simply connectedness ensures automatically the existence of a simply connected universal cover, so that the fundamental group $\pi_1(X)$ and the revised fundamental group $\bar{\pi}_1(X)$ coincide. On the other hand both the semi-local simple connectedness and the non-branching assumptions are not stable under $mGH$-convergence, while the $\RCD^{*}(K,N)$ condition is.

A second application is the extension to compact $\RCD^{*}(0,N)$-spaces of a celebrated result of Cheeger-Gromoll  \cite{ChGr} about the fundamental group of a  compact manifold  with non-negative Ricci curvature. The extension seems new even for Ricci-limit spaces. 

\begin{theorem}[Theorem \ref{thm:ChGrpi1}]\label{thm:ChGrpi1Intro}
Let $(X,\sfd,\mm)$ be a compact  $\RCD^{*}(0,N)$-space for some $N \in (1,\infty)$. Then  the revised fundamental group $\bar{\pi}_1(X)$ contains a finite normal subgroup $\psi \lhd \bar{\pi}_1(X)$ such that $ \bar{\pi}_1(X)/\psi$ is a finite group extended by $\underbrace{\Z\oplus \ldots \oplus \Z}_{k \text{ times}}$ and the universal cover $(\tilde{X},\tilde{\sfd}, \tilde{\mm})$ splits isomorphically as m.m.s. as  
$$(\tilde{X}, \tilde{\sfd}, \tilde{\mm}) \simeq (\bar{X}\times \R^{k}, \sfd_{\bar{X}\times \R^{k}}, \mm_{\bar{X}\times \R^{k}}),$$ 
where $(\bar{X}, \sfd_{\bar{X}}, \mm_{\bar{X}})$ is a \emph{compact}   $\RCD^{*}(0,N-k)$-space.
\end{theorem}
Let us mention that  the proof  follows the main steps of  \cite{ChGr}  (a fundamental tool is the Splitting Theorem  \ref{thm:Split} which was proved by  Cheeger-Gromoll \cite{ChGr} for smooth manifolds, by Cheeger-Colding \cite{Cheeger-Colding96} for Ricci limits, and by Gigli \cite{Gigli13}  in the $\RCD^{*}(0,N)$ setting)  but in a few  technical points the arguments are slightly adjusted to the non-smooth setting.
A remarkable consequence of the previous result is the following rigidity statement; to this aim denote by $\lfloor N \rfloor$ the integer part of a real number $N\in (1,\infty)$.
\begin{corollary}[Corollary \ref{cor:Rigid}] 
Let $(X,\sfd,\mm)$ be a compact  $\RCD^{*}(0,N)$-space for some $N \in (1,\infty)$. If the revised fundamental group $\bar{\pi}_{1}(X)$ contains $\lfloor N \rfloor$ independent generators of infinite order, then $(X,\sfd,\mm)$ is isomorphic as m.m.s. to an $N$-dimensional flat Riemannian manifold, i.e. there exists a covering space $(\hat{X},\hat{\sfd},\hat{\mm})$ of $(X,\sfd,\mm)$ which is isomorphic as m.m.s. to a flat torus ${\mathbb T}^{ \lfloor N \rfloor }=\R^{ \lfloor N \rfloor }/\Gamma$, for some lattice $\Gamma\subset \R^{ \lfloor N \rfloor }$.
\end{corollary}

A third application is the extension to $\RCD^*(0,N)$-spaces of Milnor's result   \cite{Mil} about the polynomial growth of finitely generated subgroups of the fundamental group of a non-compact manifold with non-negative Ricci curvature.

\begin{theorem}[Theorem \ref{thm:Milnor}]
Let $(X,\sfd, \mm)$ be an $\RCD^*(0,N)$-space for some $N\in(1,\infty)$. Then any finitely generated subgroup of the revised fundamental group $\bar{\pi}_1(X)$ has polynomial growth of degree at most $N$.
\end{theorem}

Building on the proof of the last result, one can extend to   $\RCD^*(0,N)$-spaces also Anderson's Theorem \cite{An} about maximal volume growth and finiteness of the  revised
fundamental group.

\begin{theorem}[Theorem \ref{thm:Anderson}]
Let $(X,\sfd,\mm)$ be an $\RCD^*(0,N)$-space for some $N \in (1,\infty)$ and assume that it has Euclidean volume growth, i.e. $\liminf_{r\to +\infty} \mm(B_{x_{0}}(r))/r^N=C_{X}>0$. Then the revised fundamental group is finite. More precisely, it holds  $$|\bar{\pi}_1(X,x_0)|\leq \tilde{\mm}(B_{1}(\tilde{x}))/C_{X},$$
where $(\tilde{X}, \tilde{\sfd}, \tilde{\mm})$ is the universal cover of $(X,\sfd,\mm)$ and $\tilde{x}\in \tilde{X}$ is a fixed reference point.
 \end{theorem}

A last application is the extension to $\RCD^{*}(0,N)$-spaces of a result of Sormani  \cite{So2}, about loops at infinity.  In order to state it, recall that a length space $(X,\sfd)$ has the loop to infinity property if the following holds: for any element $g\in \bar{\pi}_1(X, x_0)$ and any compact subset $K\subset \subset X$, $g$ has a representative loop of the form 
$\sigma \circ \gamma \circ \sigma^{-1}$ where $\gamma$ is a loop in $X\setminus K$ and $\sigma$ is a curve from $x_0$ to $X\setminus K$.

\begin{theorem}[Theorem \ref{thm:Sormani}]
Let $(X,\sfd,\mm)$ be a non-compact $\RCD^*(0,N)$-space for some $N \in (1, \infty)$. Then either $X$ has the loops at infinity property or the universal cover $(\tilde{X}, \sfd_{\tilde{X}},\mm_{\tilde{X}})$ splits isomorphically as metric measure space, i.e. it is isomorphic to a product  $(Y\times \R, \sfd_{Y\times \R}, \mm_{Y\times \R})$. 
\end{theorem}

Let us stress that the ones above are the first topological results on $\RCD^{*}(K,N)$-spaces without additional assumptions; nevertheless let us mention that in  \cite{BS2010} and \cite{KitLak}, under the additional non-branching and semi-local simple connectedness assumptions
(the last one automatically ensuring the existence of a simply connected universal cover), interesting results about the fundamental group have been achieved.

\begin{remark}[Open problems]
We stated a number of applications of Theorem \ref{thm:UnCovRCD-noncompactIntro}, but we do not expect of having given an exhaustive list; for instance we expect (some of) the results proved by Shen-Sormani \cite{ShenSormani} and Sormani \cite{So00} to be extendable to $\RCD^{*}(K,N)$ spaces.  Another question which would  be interesting to investigate is the existence (or not) of the universal cover for a (essentially) non-branching  $\CD^{*}(K,N)$ space, since such a condition satisfies the local-to-global property (see \cite{BS2010} and \cite{CMMap}).
\end{remark}

The paper is organized as follows. In Section 2 we review the basic definitions and properties of the various covering spaces and of spaces with lower Ricci curvature bounds  which we will need throughout the paper.  
\\In Section 3 we prove the main Theorem  \ref{thm:UnCovRCD-noncompactIntro} in the case where the space $(X,\sfd)$ is compact  and we establish the fist two applications (Theorem \ref{thm:finiteRevIntro}, Theorem \ref{thm:ChGrpi1Intro}). The proof in the compact case presents almost all the geometric ideas but is slightly less technical than the corresponding non-compact one, this is the main reason we decided to present both.
\\Finally in Section 4 we prove Theorem  \ref{thm:UnCovRCD-noncompactIntro} in full generality and discuss the applications above.
\\

{\bf Acknowledgement}:   This work was done while both  authors were in residence at the 
Mathematical Sciences Research Institute in Berkeley, California during the 
Spring 2016 semester,  supported by the National Science Foundation
under Grant No. DMS-1440140.  We thank the organizers of the Differental Geometry Program and  MSRI for providing great environment for research and collaboration.

\section{Preliminaries}
Throughout the paper $(X,\sfd)$ will be a complete and separable metric space. We will also always assume $(X,\sfd)$ to  be geodesic (i.e. any couple of points is joined by a length minimizing geodesic; actually any proper complete length space is geodesic) and  proper (i.e. closed metric balls are compact; the properness actually, even if not assumed, will follow in any case by the $\CD^{*}(K,N)$ condition, i.e. from the lower bounds on the Ricci curvature and the upper bounds on the dimension). 
We will endow the metric space $(X,\sfd)$ with a $\sigma$-finite Borel positive measure $\mm$; the triple $(X,\sfd,\mm)$ will be called  metric measure space, m.m.s. for short. In order to avoid trivialities we will always assume $\mm(X)>0$.

\subsection{Covering spaces of metric spaces}\label{SS:CovSpaces}
We start by recalling the definition of covering of a metric space \cite[Page 62,80]{Sp}.
\begin{definition}[Universal cover of  a metric space]
Let $(X,\sfd_X)$ be a metric space. We say that the metric space $(Y,\sfd_Y)$ is a \emph{covering space} for $(X, \sfd_X)$ if  there exists a continuous map $\pi: Y \to X$ such that for every point $x \in X$ there exists a neighborhood $U_x \subset X$ with the following  property:  $\pi^{-1}(U_x)$ is  the disjoint union of open subsets of $Y$ each of which is mapped  homeomorphically onto $U_x$ via $\pi$.

We say that a connected metric space $(\tilde{X}, \sfd_{\tilde{X}})$ is a \emph{universal cover} for  $(X,\sfd_X)$ if $(\tilde{X}, \sfd_{\tilde{X}})$ is a covering space  for   $(X, \sfd_X)$ with the following property: for any other covering space $(Y,\sfd_Y)$ of $(X,\sfd_X)$ there exists a continuous map $f:\tilde{X}\to Y$ such that the triangle made with the projection maps onto $X$ commutes.
\end{definition}

Given $x \in X$, we denote with $\pi_1(X,x)$ the fundamental group of $X$ based at $x$. Recall that a metric space $(X,\sfd)$ is said  \emph{semi-locally simply connected} (or semi-locally one connected) if for all $x \in X$ there is a neighborhood $U_x$ of $x$ such that any curve in $U_x$ is contractible in $X$, i.e. $\pi_1(U_x,x)\to  \pi_1(X,x)$ is trivial (cf. \cite[p. 78]{Sp}, \cite[p. 142]{Mas}).  This is weaker than saying that $U_x$ is simply connected.

\begin{remark}[Existence of a universal cover]\label{rem:ExistUC}
A universal cover may not exist in general \cite[Ex. 17, p. 84]{Sp}, however if it exists then it is unique. Furthermore, if a space is locally path connected and semi-locally simply connected
then it has a universal cover and that cover is simply connected \cite[Cor. 14, p. 83]{Sp}. On the other hand, the universal covering space of a locally path connected space may not be simply connected  \cite[Ex. 18, p. 84]{Sp}.
\end{remark}

If a space has a universal cover, then we can consider the revised fundamental group  introduced in \cite{SW1},  which is not as fine as the standard fundamental group but still can give interesting topological information in many situations.
\\First of all we denote by $G(Y,X)$ the \emph{group of deck transforms} (or self equivalences) of a cover $\pi:Y\to X$. This is by definition the
group of homeomorphisms $h: Y \to Y$ such that $\pi\circ h=\pi$ (cf. \cite[p. 85]{Sp}).  

\begin{definition}[Revised fundamental group]
Assume the metric space $(X,\sfd_{X})$ admits a universal cover  $(\tilde{X}, \sfd_{\tilde{X}})$. Then, \cite[Def. 2.4]{SW1},  the \emph{revised fundamental group} of $X$, denoted  with $\bar\pi_1(X)$, is the  group of deck transformations $G(\tilde{X}, X)$.  
\end{definition}

\begin{remark}
\begin{enumerate}
\item Let $\pi:Y\to X$ be a covering and  fix  $p \in X$. Then, for any fixed $\tilde{p}\in \pi^{-1}(p) \subset Y$,  there is a natural surjection $\Psi_{\tilde{p}}: \pi_1(X,p) \to G(Y,X)$ defined as follows.

Note that given $[\gamma] \in \pi_1(X,p)$, we can lift the loop $\gamma$ to a curve based at $\tilde{p}$ in the cover $Y$.  This defines an action of $[\gamma]$ on $\pi^{-1}(p) \subset Y$ which can
be extended uniquely to a deck transform. This map is surjective when $Y$ is path connected because given any $h \in G(Y,X)$ we can join $\tilde{p}$ to $h(\tilde{p})$ by a curve which can be projected to the base space giving an element of  $\pi_1(X,p)$.  
The kernel $H_p$ of $\Psi_{\tilde{p}}: \pi_1(X,p) \to G(Y,X)$, consists of those elements of the fundamental
group  $\pi_1(X,p)$, whose representative loops are still closed when they are lifted to the cover.

\item  $\bar\pi_1(X)$ is canonically isomorphic to $\pi_1(X,p)/H_p$, where $H_p$ is the kernel  of $\Psi_{\tilde{p}}: \pi_1(X,p) \to  G(\tilde{X}, X)$. 

\item When the universal cover is simply connected and locally path connected, then $\bar\pi_1(X)$ is canonically isomorphic to $\pi_1(X,p)$, see \cite[Cor. 4, p. 87]{Sp}.  
\end{enumerate}
\end{remark}

Let us recall that (cf. \cite[p. 81]{Sp}), given an open covering ${\mathcal U}$ of $X$, there exists a covering space $\tilde{X}_{\mathcal U}$ such that $\pi_1(\tilde{X}_{\mathcal U}, \tilde{p})$  is isomorphic to $\pi_1(X,{\mathcal U},p)$, where $\tilde{p} \in \pi^{-1} (p)$ and  $\pi_1(X,{\mathcal U},p)$ is the  normal subgroup of $\pi_1(X, p)$  generated
by homotopy classes of closed paths having a representative of the form $\alpha^{-1} \circ \beta \circ \alpha$, where $\beta$ is a
closed path lying in some element of ${\mathcal U}$ and $\alpha$ is a path from $p$ to $\beta(0)$.  In the following, $\pi_1(\tilde{X}_{\mathcal U}, \tilde{p})$ will be called \emph{covering group} and will be  identified with the normal subgroup  $\pi_1(X,{\mathcal U},p)\subset \pi_1(X,p)$.

We also mention that, given a length space $(X,\sfd)$, if the universal cover $\tilde{X}$ exists, then it  is obtained as   covering space $\tilde{X}_{\mathcal{U}}$ associated to a suitable  open cover $\mathcal{U}$ of $X$ satisfying the following property: for every $x\in X$ there exists $U_{x}\in \mathcal{U}$ such that $U_{x}$ is lifted  homeomorphically by any covering space of $(X,\sfd)$.

We now recall the notion of $\delta$-cover introduced in \cite{SW1}. Such objects are a sort of filter at scale $\delta$ for the first fundamental group and are useful in order to investigate the existence of the universal cover. Indeed since in the paper we do not want to assume  semi-local simply connectedness of the spaces, the first non trivial issue in studying the fundamental group  is exactly the existence or not  of  a universal cover.


\begin{definition}[$\delta$-cover]
Let $(X, \sfd)$ be a length metric space and fix $\delta>0$. The \emph{$\delta$-cover} of $X$, denoted  by $\tilde{X}^\delta$, is defined to be $\tilde{X}_{\mathcal U_\delta}$, where $\mathcal{U_{\delta}}$ is the open cover of $X$ consisting of all the balls of radius $\delta$.  The covering group $\pi_1(\tilde{X}_{\mathcal U_\delta}, \tilde{p})$ will be denoted with $\pi_1(X, \delta, p) \subset \pi_1(X,p)$, and the group of deck transformations of $\tilde{X}^\delta$ will be denoted by $G(X,\delta)= \pi_1(X,p)/\pi_1(X, \delta, p)$.
\end{definition}

\begin{remark}
\begin{enumerate}
\item  If $\delta_1\leq \delta_2$ then $\tilde{X}^{\delta_1}$ covers $\tilde{X}^{\delta_2}$. 
\item  The group of deck transformations $G(X, \delta)$,  does not depend on $p$ \cite[Cor. 3, p. 86]{Sp}. One can
think of $G(X, \delta)$ as roughly corresponding to those loops which are not generated by short loops (i.e. of length at most $2\delta$)  in $\pi_1(X,p)$. More precisely, one can think of the $\delta$-cover as an intermediate cover which only unwraps those loops which are not contained in any ball of radius $\delta$.   
\end{enumerate}
\end{remark}

The following result proved by C. Sormani and the second author \cite[Proposition 3.2, Theorem 3.7]{SW1} will be a key technical tool in order to investigate topological properties of $\RCD^*(K,N)$-spaces.

\begin{theorem}[{\cite[Proposition 3.2, Theorem 3.7]{SW1}}]\label{thm:SWUniCov}
Let $(X,\sfd)$ be a compact length metric space. Then $X$ admits  a universal cover if  and only if the $\delta$-covers of $X$ stabilize for small $\delta$, i.e. if there exists $\delta_0>0$ such that $\tilde{X}^\delta$ is isometric to 
$\tilde{X}^{\delta_0}$ for every $\delta\in (0,\delta_0]$;  moreover, in this case, the universal cover is isometric to $\tilde{X}^{\delta_0}$. 
\end{theorem}

For non-compact length spaces, the universal cover may not be a $\delta$-cover since loops could be homotopic to arbitrary small one at infinity. A simple  example is a cylinder with one side pinched to a cusp.  A natural way is to consider bigger and bigger balls. On the other hand, the universal cover of a ball may not exist even if the universal cover of the whole space exists. Also one would like to stay away from the boundaries of balls. For this purpose,  the second author  and C. Sormani
introduced relative $\delta$-covers \cite{SW2}.

We will use the following
convention. Open balls are denoted by $B_R(x)$ while closed balls are denoted by $B(x,R)$ all with intrinsic metric.

\begin{definition}[Relative $\delta$-cover]  Suppose $X$ is a length space, $x \in X$ and $0<r<R.$
	Let \[\pi^{\delta} : \tilde{B}_R(x)^{\delta} \to B_R(x)\] be the
	$\delta$-cover of the open ball $B_R(x).$  A connected component
	of $(\pi^{\delta})^{-1}(B(x,r)),$ where $B(x,r)$ is a closed ball,
	is called a relative $\delta$-cover of $B(x,r)$ and is denoted by
	$\tilde{B}(x,r,R)^{\delta}.$
\end{definition}

%
%
%
%
%
\noindent
We will make use of the following simple lemmas from \cite{SW2}. 

\begin{lemma} \label{distance}
	The restricted
	metric on $B(p,R)$ from $X$ is the same as the intrinsic metric on
	$B(p,2R+\epsilon)$ restricted to $B(p,R)$ for any $\epsilon >0$.
	Namely,
	\be \sfd_X(q_1, q_2)=\sfd_{B(p,2R+\epsilon)}(q_1,q_2),
	\qquad \forall q_1, q_2 \in B(p,R).
	\ee
\end{lemma}

\begin{lemma} \cite[Lemma 2.3]{SW2}  \label{localisom}
	The covering map
	\be
	\pi^\delta (r,R): (\tilde{B}(p,r,R)^\delta,
	d_{\tilde{B}(p,r,R)^\delta}) \rightarrow (B(p,r), d_{B(p,r)})
	\ee
	is an isometry on balls of radius $\delta/3$.
\end{lemma}

Instead of Theorem~\ref{thm:SWUniCov} which was the key to  prove the existence of the universal cover for a compact length space,  for non-compact spaces the  key role will be played by the following result   \cite[Lemma 2.4, Theorem 2.5]{SW2}.  See \cite{Ennis-Wei2010} for a description of the universal cover. 

\begin{theorem} \label{univstab}
	Let  $(X,\sfd)$ be a length space and assume that there is $x\in X$ with the following property: for
	all $r>0$, there exists $R \ge r$, such that $\tilde{B}(x,r,R)^\delta$
	stabilizes for all $\delta$ sufficiently small.  Then $(X,\sfd)$ admits a  universal cover $\tilde{X}$. 
	\end{theorem}

\subsection{Lift of metric measure spaces}\label{ss:LiftMMS}
In this short section we briefly recall how to lift the metric-measure structures of the space $(X,\sfd,\mm)$ to a covering space $\tilde{X}_{\mathcal U}$ associated to an open cover ${\mathcal U}$ of $X$ (so in particular to a $\delta$ cover $\tilde{X}^{\delta}$); let us mention that a similar construction for the universal cover was performed in \cite{BS2010}.  First of all if $(X,\sfd)$ is a   \emph{locally compact length}  space then the  covering space  $\tilde{X}_{\mathcal U}$  inherits the locally-compact length structure of the base $X$ in the following way. Denote with $\pi: \tilde{X}_{\mathcal U}  \to X$ the projection map and let us call  ``admissible''  a curve $\tilde{\gamma}$ in $\tilde{X}_{\mathcal U}$ if and  only if its composition with $\pi$ is a continuous curve in $(X,\sfd)$.   The length $L_{\tilde{X}_{\mathcal U}} (\tilde{\gamma})$ of an admissible curve in $\tilde{X}_{\mathcal U}$ is set to be the length of  $\pi \circ \tilde{\gamma}$ with respect to the
length structure in $(X,\sfd)$. For two points $\tilde{x},\tilde{y} \in \tilde{X}_{\mathcal U}$ we define the associated distance $\sfd_{\tilde{X}_{\mathcal U}}(\tilde{x},\tilde{y})$  to be the infimum of lengths of admissible curves in $\tilde{X}_{\mathcal U}$ connecting them:
\be\label{eq:deftildedX}
\sfd_{\tilde{X}_{\mathcal U}} (\tilde{x}, 	\tilde{y}):= \inf \left\{ L_{\tilde{X}_{\mathcal U}}(\tilde{\gamma}) \, | \, \tilde{\gamma}:[0,1]\to \tilde{X}_{\mathcal U} \text{ is admissible and } \tilde{\gamma}(0)=\tilde{x}, \tilde{\gamma}(1)=\tilde{y}  \right\}.
\ee
It is readily checked that $\tilde{\pi}: (\tilde{X}_{\mathcal U}, \sfd_{\tilde{X}_{\mathcal U}}) \to (X,\sfd)$ is a local isometry. 
In order to construct a lift of the measure $\mm$ to $\tilde{X}_{\mathcal U}$ let us consider the family
$$
\tilde{\Sigma}_{{\mathcal U}}:=\left\{ \tilde{E}\subset \tilde{X}_{\mathcal U} \,|\, {\pi}_{|\tilde{E}}: \tilde{E}\to E:=\pi(\tilde{E}) \text{ is an isometry} \right\}.
$$
The  family $\tilde{\Sigma}_{\mathcal U}$ is  clearly stable under intersections, therefore
the smallest $\sigma$-algebra $\sigma(\tilde{\Sigma}_{\mathcal U})$ containing  $\tilde{\Sigma}_{\mathcal U}$ coincides with the Borel $\sigma$-algebra ${\mathcal B}(\tilde{X}_{{\mathcal U}})$ according to the local compactness of $(\tilde{X}_{{\mathcal U}}, \sfd_{\tilde{X}_{\mathcal U}})$. 
We can then define a function $\mm_{\tilde{X}_{\mathcal U}}: \tilde{\Sigma}_{\mathcal U}\to [0, \infty)$ by setting $\mm_{\tilde{X}_{{\mathcal U}}}(\tilde{E}):=\mm\big(\pi(\tilde{E})\big)=\mm(E)$ and extend it in a unique way to a measure $\mm_{\tilde{X}_{{\mathcal U}}}$  on $\big(\tilde{X}_{{\mathcal U}},  {\mathcal B}(\tilde{X}_{{\mathcal U}})\big). $
\\The metric measure space $(\tilde{X}_{\mathcal U}, \sfd_{\tilde{X}_{\mathcal U}}, \mm_{\tilde{X}_{\mathcal U}})$ is called the ${\mathcal U}$-covering space or simply the  ${\mathcal U}$-lift  of the metric measure space $(X,\sfd,\mm)$.

\subsection{$\RCD^*(K,N)$-spaces}  \label{RCD_intro}

In this section  we quickly recall those basic definitions and properties of spaces with lower Ricci curvature bounds that we will need later on.

We denote by ${\mathcal  P}(X)$ the space of Borel probability measures on the complete and separable metric space $(X,\sfd)$ and by ${\mathcal  P}_{2}(X) \subset {\mathcal  P}(X)$ the subspace consisting of all the probability measures with finite second moment.

For $\mu_0,\mu_1 \in {\mathcal  P}_{2}(X)$ the quadratic transportation distance $W_2(\mu_0,\mu_1)$ is defined by
\begin{equation}\label{eq:Wdef}
  W_2^2(\mu_0,\mu_1) = \inf_\sggamma \int_X \sfd^2(x,y) \,\d\ggamma(x,y),
\end{equation}
where the infimum is taken over all $\ggamma \in {\mathcal  P}(X \times X)$ with $\mu_0$ and $\mu_1$ as the first and the second marginal.

Assuming the space $(X,\sfd)$ to be geodesic,  then the space $({\mathcal  P}_{2}(X), W_2)$ is also geodesic. We denote  by $\geo(X)$ the space of (constant speed minimizing) geodesics on $(X,\sfd)$ endowed with the $\sup$ distance, and by $\e_t:\geo(X)\to X$, $t\in[0,1]$, the evaluation maps defined by $\e_t(\gamma):=\gamma_t$. It turns out that any geodesic $(\mu_t) \in \geo({\mathcal  P}_{2}(X))$ can be lifted to a measure $\ppi \in {\mathcal  P}(\geo(X))$, so that $(\e_t)_\#\ppi = \mu_t$ for all $t \in [0,1]$. Given $\mu_0,\mu_1\in{\mathcal  P}_{2}(X)$, we denote by $\gopt(\mu_0,\mu_1)$ the space of all
$\ppi \in {\mathcal  P}({\geo(X)})$ for which $(\e_0,\e_1)_\#\ppi$ realizes the minimum in \eqref{eq:Wdef}. If $(X,\sfd)$ is geodesic, then the set $\gopt(\mu_0,\mu_1)$ is non-empty for any $\mu_0,\mu_1\in{\mathcal  P}_{2}(X)$.

We turn to the formulation of the $\CD^*(K,N)$ condition, coming from  \cite{BS2010}, to which we also refer for a detailed discussion of its relation with the $\CD(K,N)$ condition
 (see also \cite{Cavalletti12} and \cite{CavaSturm}). Here let just mention that the   $\CD^*(K,N)$ condition is  a priori weaker than the $\CD(K,N)$, and the two coincide  for $K=0$.  On the other hand  most of the comparison theorems hold already in sharp form  in (essentially non-branching) $\CD^{*}(K,N)$-spaces (Bishop-Gromov volume comparison \cite{CavaSturm}, L\'evy-Gromov  isoperimetric inequality \cite{CM1}, Brunn-Minkowski inequality \cite{CM2}), and moreover (non-branching) $\CD^{*}(K,N)$ satisfies the local-to-global property, a fact which is still not completely understood for (non-branching) $\CD(K,N)$ despite remarkable partial results   (see for instance \cite{Cavalletti12},\cite{CM2}, \cite{CavaSturm},  \cite{Rajala}).

Given $K \in \R$ and $N \in [1, \infty)$, we define the distortion coefficient $[0,1]\times\R^+\ni (t,\theta)\mapsto \sigma^{(t)}_{K,N}(\theta)$ as
\[
\sigma^{(t)}_{K,N}(\theta):=\left\{
\begin{array}{ll}
+\infty,&\qquad\textrm{ if }K\theta^2\geq N\pi^2,\\
\frac{\sin(t\theta\sqrt{K/N})}{\sin(\theta\sqrt{K/N})}&\qquad\textrm{ if }0<K\theta^2 <N\pi^2,\\
t&\qquad\textrm{ if }K\theta^2=0,\\
\frac{\sinh(t\theta\sqrt{K/N})}{\sinh(\theta\sqrt{K/N})}&\qquad\textrm{ if }K\theta^2 <0.
\end{array}
\right.
\]
\begin{definition}[Curvature dimension bounds]
Let $K \in \R$ and $ N\in[1,  \infty)$. We say that a m.m.s.  $(X,\sfd,\mm)$
 is a $\CD^*(K,N)$-space if for any two measures $\mu_0, \mu_1 \in {\mathcal  P}(X)$ with support  bounded and contained in $\supp(\mm)$ there
exists a measure $\ppi \in \gopt(\mu_0,\mu_1)$ such that for every $t \in [0,1]$
and $N' \geq  N$ we have
\begin{equation}\label{eq:CD-def}
-\int\rho_t^{1-\frac1{N'}}\,\d\mm\leq - \int \sigma^{(1-t)}_{K,N'}(\sfd(\gamma_0,\gamma_1))\rho_0^{-\frac1{N'}}+\sigma^{(t)}_{K,N'}(\sfd(\gamma_0,\gamma_1))\rho_1^{-\frac1{N'}}\,\d\ppi(\gamma), 
\end{equation}
where for any $t\in[0,1]$ we  have written $(\e_t)_\sharp\ppi=\rho_t\mm+\mu_t^s$  with $\mu_t^s \perp \mm$.
\end{definition}
Notice that if $(X,\sfd,\mm)$ is a $\CD^*(K,N)$-space, then so is $(\supp(\mm),\sfd,\mm)$, hence it is not restrictive to assume that $\supp(\mm)=X$, a hypothesis that we shall always implicitly do from now on. 

On $\CD^*(K,N)$-spaces a natural version of the Bishop-Gromov volume growth estimate holds (see \cite[Theorem 6.2]{BS2010} and \cite{CavaSturm}), however we will use just the following weaker statements:
\begin{theorem}
Let  $K\in \R$, $N\geq 1$ be fixed. Then there exists a function $\Lambda_{K,N}(\cdot, \cdot): \R_{>0} \times \R_{>0} \to \R_{>0}$ such that if $(X,\sfd,\mm)$ is a $\CD^*(K, N)$ space for some $K\in \R$, $N\geq 1$ the following holds:
\be\label{eq:BishGrom}
\frac{\mm(B_r(x))}{\mm(B_R(x))} \geq \Lambda_{K,N}(r,R), \quad \forall 0<r\leq R< 	\infty.
\ee
For $K=0$  the following more explicit bound holds
\be\label{eq:BGK0}
\frac{\mm(B_r(x))}{\mm(B_R(x))} \geq \left(\frac{r}{R}\right)^{N}, \quad \forall 0<r\leq R< 	\infty.
\ee

\end{theorem}

One crucial property of the $\CD^*(K,N)$ condition is the stability under measured Gromov Hausdorff convergence of m.m.s., so that Ricci limit spaces are $\CD^{*}(K,N)$. Moreover, on the one hand it is possibile to see that Finsler manifolds are allowed as $\CD^{*}(K,N)$-space while on the other hand from the work of Cheeger-Colding \cite{Cheeger-Colding97I},\cite{Cheeger-Colding97II},\cite{Cheeger-Colding97III}  it was understood that purely Finsler structures never appear as Ricci limit spaces.   Inspired by this fact, in \cite{Ambrosio-Gigli-Savare11b}, Ambrosio-Gigli-Savar\'e proposed a  strengthening of the $\CD$ condition in order to enforce, in some weak sense, a Riemannian-like behavior of spaces with a curvature-dimension bound (to be precise in  \cite{Ambrosio-Gigli-Savare11b} it was analyzed the case of strong-$\CD(K,\infty)$ spaces endowed with a probability reference measure $\mm$; the axiomatization has been then simplified and generalized in  \cite{AmbrosioGigliMondinoRajala} to allow $\CD(K,\infty)$-spaces endowed with a $\sigma$-finite reference measure); the finite dimensional refinement led to the so called  $\RCD^{*}(K,N)$ condition, for $N\in (1,\infty)$.  
Such a strengthening consists in requiring that the space $(X,\sfd,\mm)$ is such that the Sobolev space $W^{1,2}(X,\sfd,\mm)$ is Hilbert, a condition we shall refer to as `infinitesimal Hilbertianity'. It is out of the scope of this note to provide full details about the definition of $W^{1,2}(X,\sfd,\mm)$ and its relevance in connection with Ricci curvature lower bounds. We will instead be satisfied in recalling the definition and few crucial properties which are relevant for our discussion: the stability under convergence of m.m.s. (see \cite{GigliMondinoSavare} and references therein), the Abresh-Gromoll Inequality (\cite{GigliMosconi} and \cite[Theorem 3.7, Corollary 3.8]{MNRect}),   the Splitting Theorem (see \cite{Gigli13}), and the a.e. infinitesimal regularity (see \cite{GigliMondinoRajala} and \cite{MNRect}). We also wish to mention that the $\RCD^{*}(K,N)$ condition is equivalent to the $N$-dimensional Bochner inequality, as proved independently in \cite{EKS} and  \cite{AMS}.

First of all recall that on a m.m.s. there is a canonical notion of `modulus of  the differential of a function' $f$, called weak upper differential and denoted with $|Df|_w$; with this object one defines the Cheeger energy
$$\Ch(f):=\frac 1 2 \int_X |Df|_w^2 \, \d \mm.$$
The  Sobolev space $W^{1,2}(X,\sfd,\mm)$ is by definition  the space of $L^2(X,\mm)$ functions having finite Cheeger energy, and it is endowed with the natural norm  $\|f\|^2_{W^{1,2}}:=\|f\|^2_{L^2}+2 \Ch(f)$ which makes it a Banach space. We remark that, in general, $W^{1,2}(X,\sfd,\mm)$ is not Hilbert (for instance, on a smooth Finsler manifold the space $W^{1,2}$ is Hilbert if and only if the manifold is actually Riemannian); in case  $W^{1,2}(X,\sfd,\mm)$ is  Hilbert then we say that $(X,\sfd,\mm)$ is \emph{infinitesimally Hilbertian}. 

\begin{definition}
An $\RCD^{*}(K,N)$-space $(X,\sfd,\mm)$ is an infinitesimally Hilbertian $\CD^*(K,N)$-space.
\end{definition}

Now we state a few fundamental properties of $\RCD^*(K,N)$ spaces.  First of all, on $\RCD^*(K,N)$-spaces a natural version of the Abresch-Gromoll inequality \cite{AG90} holds (see \cite{GigliMosconi} and \cite[Theorem 3.7, Corollary 3.8]{MNRect}). Here let us just recall  the following statement which is a particular case of \cite[Corollary 3.8]{MNRect} and will be enough for our purposes.
\begin{theorem}[Abresch-Gromoll Inequality]  \label{AG-ineq}
Given $K\in \R$ and $N\in(1,+\infty)$ there exist $\alpha(N)\in (0,1)$ and $C(K,N)>0$ with the following properties.
Given $(X,\sfd,\mm)$ an $\RCD^*(K,N)$-space, fix $p,q \in X$ with $\sfd_{p,q}:=\sfd(p,q)\leq 1$ and let $\gamma$ be a constant speed minimizing geodesic from $p$ to $q$. Then
\be\label{eq:AG}
e_{p,q}(x)\leq C(K,N) r^{1+\alpha(N)} \sfd_{p,q}, \quad \forall x \in B_{r \, \sfd_{p,q}}(\gamma(1/2)), \ \ r \in (0, \tfrac 14), 
\ee
where $e_{p,q}(x):=\sfd(p,x)+\sfd(x,q)-\sfd(p,q)$ is the so called excess function associated to $p,q$.
\end{theorem} 

Another fundamental property of $\RCD^{*}(0,N)$-spaces is the extension of the celebrated  Cheeger-Gromoll Splitting Theorem  \cite{ChGr} proved in \cite{Gigli13}   (let us also mention that for Ricci limit spaces the Splitting Theorem was established by Cheeger-Colding \cite{Cheeger-Colding96}).

\begin{theorem}[Splitting]\label{thm:Split} 
 Let $(X,\sfd,\mm)$ be an  $\RCD^*(0,N)$-space with $1 \le N < \infty$. Suppose that
 $X$ contains a line. Then $(X,\sfd,\mm)$ is isomorphic to $(X'\times \R, \sfd'\times \sfd_E,\mm'\times \LL_1)$,
where $\sfd_E$ is the Euclidean distance, $\LL_1$ the Lebesgue measure and  $(X',\sfd',\mm')$ is an $\RCD^*(0,N-1)$-space if $N \ge 2$ and a singleton if $N < 2$.
\end{theorem}

Here, by `line',  we intend an isometric embedding of $\R$. 
In the paper we will need also the existence of an infinitesimally regular point, i.e. a point where the tangent cone is unique and isometric  to $\R^{n}, n\leq N$.   The a.e. infinitesimal regularity was settled for Ricci-limit spaces by Cheeger-Colding  \cite{Cheeger-Colding97I},\cite{Cheeger-Colding97II},\cite{Cheeger-Colding97III}; for an  $\RCD^{*}(K,N)$-space $(X,\sfd,\mm)$,  it was proved  in \cite{GigliMondinoRajala} that for $\mm$-a.e. $x\in X$ there exists a blow-up sequence converging to a Euclidean space. The $\mm$-a.e.  uniqueness of the blow-up limit, together with the rectifiability of an $\RCD^{*}(K,N)$-space, was then established in \cite{MNRect}.  More precisely the following  holds (actually the result in full strength proved in \cite{GigliMondinoRajala},\cite{MNRect} is more precise as it includes also the convergence of the rescaled measures, but  here let  state this shorter version which will suffice for the present work).

\begin{theorem}[Infinitesimal regularity of $\RCD^{*}(K,N)$-spaces]  \label{RCD-reg}
Let $(X,\sfd,\mm)$ be an $\RCD^{*}(K,N)$-space for some $K\in \R, N \in (1,\infty)$. Then $\mm$-a.e. $x\in X$ is a regular  point, i.e. for $\mm$-a.e.  $x \in X$ there exists $n=n(x)\in [1,N]\cap \N$ such that, for any sequence $r_{i}\downarrow 0$,  the rescaled  pointed metric spaces $(X, r_{i}^{-1} \sfd, x)$ converge in the pointed Gromov-Hausdorff sense to the pointed Euclidean space $(\R^{n}, \sfd_E, 0^{n})$. 
\end{theorem}
We end the section with the next lemma which will be useful in the sequel.

\begin{lemma}\label{lem:RCDcover}
Let $(X,\sfd,\mm)$ be an $\RCD^*(K,N)$-space for some $K\in \R, N\geq 1$ and let ${\mathcal U}$ be an open cover of $X$. Then the  ${\mathcal U}$-covering space $(\tilde{X}_{\mathcal U}, \sfd_{\tilde{X}_{\mathcal U}}, \mm_{\tilde{X}_{\mathcal U} })$ is an  $\RCD^*(K,N)$-space.
\end{lemma}

\begin{proof}
From the construction of the metric-measure structure $(\tilde{X}_{\mathcal U}, \sfd_{\tilde{X}_{\mathcal U}}, \mm_{\tilde{X}_{\mathcal U}})$ performed in  Section  \ref{ss:LiftMMS}, it is clear that each point $\tilde{x}\in \tilde{X}_{\mathcal U}$ has a neighborhood $\tilde{U}$ such that $(\tilde{U}, {\sfd_{\tilde{X}_{\mathcal U}}}_{|\tilde{U}}, {\mm_{\sfd_{\tilde{X}_{\mathcal U}}}}\llcorner {\tilde{U}}) $ is isomorphic as metric measure space to $(U, \sfd_{|U}, \mm \llcorner U)$, where of course $U:=\pi(\tilde{U})$ is a neighborhood of $x:=\pi(\tilde{x}).$
Therefore $(\tilde{X}_{\mathcal U}, \sfd_{\tilde{X}_{\mathcal U}},  \mm_{\tilde{X}_{\mathcal U}}))$ satisfies the $\RCD^*(K,N)$ condition locally and the thesis follows by the local-to-global property of $\RCD^*(K,N)$ proved independently in \cite[Theorem 7.8]{AMSLocToGlob} \cite[Theorem 3.25] {EKS}. To be more precise one can first globalize the Curvature-Dimension condition using   \cite[Theorem 3.14]{EKS} and then globalize the infinitesimal Hilbertianity by recalling that the Cheeger energy  is a local object (for more details on the Cheeger energy see \cite[Section 4.3]{Ambrosio-Gigli-Savare11b}, \cite[Section 3]{AmbrosioGigliMondinoRajala} and \cite[Section 4.3]{GigliMemo}). 
\end{proof}

\section{Compact $\RCD^*(K,N)$-spaces admit a universal cover}
The goal of this section is to prove the following result.
\begin{theorem}\label{thm:UnCovRCD}
Let $(X,\sfd, \mm)$ be a   compact  $\RCD^*(K,N)$ space for some $K\in \R$, $N\in(1, +\infty)$.  Then there exists $\delta_0>0$ such that $\tilde{X}^{\delta}$ is isometric to $\tilde{X}^{\delta_0}$, for all $\delta\leq \delta_0$; in particular, in virtue of Theorem \ref{thm:SWUniCov}, $(X,\sfd)$ admits a universal cover. 
\end{theorem}

In order to prove Theorem \ref{thm:UnCovRCD},  we first establish the following result roughly saying that the topology of $X$ stabilizes in a small neighborhood of a regular point.

\begin{theorem}\label{thm:LocStab}
Let $(X,\sfd, \mm)$ be an $\RCD^*(K,N)$ space for some $K \in \R, N\in (1,\infty)$ and let $x \in X$ be a  regular point. Then there exists $r_x>0$  such that $B_{r_x}(x)$ lifts isometrically to $\tilde{X}^\delta$ for all $\delta>0$.
\end{theorem}

\begin{remark}
It is always true that for every $x \in X$ we can find  $r_{x,\delta}>0$ such that  $B_{r_{x,\delta}}(x)$ lifts isometrically to $\tilde{X}^\delta$; the non-trivial content of Theorem \ref{thm:LocStab} is that if $x$ is a regular point, the radius $r_{x,\delta}>0$ does \emph{not} depend on $\delta$. 
\end{remark}


\begin{remark}\label{rem:simplificationsCompact}
Let us mention that the proof of Theorem  \ref{thm:LocStab} is simpler than the  proof  of the corresponding statement for limits of Riemannian manifolds (see \cite[Theorem 4.5]{SW1}): 
\begin{itemize}
\item first we argue directly on $X$ without the need to go back and forth to the smooth approximations, 
\item the second and more important simplification is that we can consider directly  a minimizing geodesic on the $\delta$-cover $\tilde{X}^\delta$ and argue via the Abresch-Gromoll excess estimate \eqref{eq:AG}, while in the Ricci-limit case  the approximation by geodesics in the smooth approximating manifolds prevented to use directly Abresch-Gromoll excess estimate and the authors had to go though the so called Uniform Cut Lemma, see \cite[Lemma 4.6]{SW1} and \cite[Lemma 7]{So00}.
\end{itemize}
\end{remark}

\emph{Proof of Theorem \ref{thm:LocStab}}
The proof is by contradiction. Fix a regular point $x \in X$ and  assume that  for every $r>0$ there exists $\delta_r>0$ such that the metric ball $B_r(x)$ does not lift isometrically to $\tilde{X}^{\delta_r}$. Then  we can find  sequences $r_i\downarrow 0, \delta_i=\delta_{r_i}$ and $g_i\in G(X,\delta_i)$ such that, fixed a lifting $\tilde{x}_{i} \in \pi_{i}^{-1}(x) \in \tilde{X}^{\delta_i}$, it holds 
\be\label{eq:ri}
\sfd_i:=\sfd_{\tilde{X}^{\delta_i}}(\tilde{x}_i, g_i \tilde{x}_i)=\min \left\{\sfd_{\tilde{X}^{\delta_i}}(\tilde{x}_i, h \tilde{x}_i) \, : \,   1\neq h\in G(X,\delta_i)  \right\} \in (0,2r_i).
\ee 
Now let  $\tilde{\gamma}_i \subset \tilde{X}^{\delta_i}$ be a  minimizing geodesic from $\tilde{x}_i$ to $g_i \tilde{x}_i$. By using \eqref{eq:ri} and imitating the proof of the Halfway Lemma \cite[Lemma 5]{So00}, 
we infer that   $\gamma_i:=\pi_i(\tilde{\gamma}_i)\subset X$ are halfway minimizing, i.e. $\sfd_X(\gamma_i(0), \gamma_i(1/2))=\sfd_i/2$.
By Gromov compactness and by the fact that $x$ is a regular point,  we can choose a subsequence of these $i$  such that the rescaled pointed spaces $(X,\frac{1}{\sfd_i} \, \sfd ,x )$  converge to the tangent cone $(\R^k, \sfd_{\R^k}, 0^k)$ in the pointed Gromov-Hausdoff sense; in other words
\be\label{eq:epsiloni}
\sfd_{GH} (B_{10 \sfd_i}(x), B_{10 \sfd_i}(0^k)) \leq \varepsilon_i \, \sfd_i, \quad \text{with } \varepsilon_i \to 0,
\ee
where of course $B_{10 \sfd_i}(0^k)\subset \R^k$ is a Euclidean metric ball.
Now the intuitive idea is that the curves $\gamma_i$ are  closed based  geodesic loops in $X$ shrinking towards $x$, so we can find  corresponding closed curves in the tangent cone  $\R^k$ which are  ``almost  closed based  geodesic loops'' ; but since  $\R^k$ has no  closed based  geodesic loop,  we expect to get a contradiction. In order to make a rigorous proof we need to make quantitative statements, to this aim we are going to use the Abresch-Gromoll inequality  \eqref{eq:AG}.  

 Let $S\in (0,1)$ be a  fixed positive number that in the end of the proof will be chosen small enough.
 Denote with $A_{r,R}(0^k):=B_R(0^k)\setminus B_r(0^k)$ the annulus in $\R^k$ centered at the origin with radii $r,R$. From \eqref{eq:epsiloni}, we get that there exist
\be\label{eq:barx}
y_i\in  A_{(1/2-\varepsilon_i) \sfd_i, (1/2+\varepsilon_i) \sfd_i}  (0^k)\subset \R^k,\,  z_i \in A_{(1/2+2S -\varepsilon_i) \sfd_i, (1/2+2S+\varepsilon_i) \sfd_i}(0^k) \subset \R^k, \,\bar{x}_i \in \partial B_{(1/2+ 2 S) \sfd_i}(x) \subset X
\ee
such that
\be\label{eq:Contr1}
\sfd (\gamma_i(1/2), \bar{x}_i) \leq \sfd_{\R^k} (y_i, z_i) + \varepsilon_i \sfd_i \leq 2\varepsilon_i \sfd_i + 2 S \sfd_i + \varepsilon_i \sfd_i.
\ee
We now lift the points $x$ and $\bar{x}_i$ to the $\delta_i$-cover $\tilde{X}^{\delta_i}$ as follows. Recall from the beginning of the proof  that the curve $\tilde{\gamma}_i\subset \tilde{X}^{\delta_i}$  goes from $\tilde{x}_i$ through $\tilde\gamma_i(1/2)$ to $g_i \tilde{x}_i$, and lifts the closed loop $\gamma_i$.
Let  $\sigma_i\subset X$ be a minimizing geodesic from $\gamma_i(1/2)$ to $\bar{x}_i$ and lift it to a minimizing geodesic $\tilde{\sigma}_i \subset \tilde{X}^{\delta_i}$ going from  $\tilde\gamma_i(1/2)$ to some point $\tilde{\bar{x}}_i \in \tilde{X}^{\delta_i}$ such that $\pi_i(\tilde{\bar{x}}_i )=\bar{x}_i$ and
\be\label{eq:distgt}
\sfd_{\tilde{X}^{\delta_i}}(\tilde{\gamma}(1/2) , \tilde{\bar{x}}_i)=\sfd(\gamma(1/2), \bar{x}_i).
\ee
 By our choice of  $\bar{x}_i$ in \eqref{eq:barx} we get
\be\label{eq:triangletilde}
\sfd_{\tilde{X}^{\delta_i}}(\tilde{x}_i, \tilde{\bar{x}}_i) \geq \sfd(x, \bar{x}_i)= \frac{\sfd_i}{2}+2S \sfd_i \quad \text{and}  \quad \sfd_{\tilde{X}^{\delta_i}}(g_i \tilde{x}_i, \tilde{\bar{x}}_i) \geq \frac{\sfd_i}{2}+2S \sfd_i.
\ee
Combining  \eqref{eq:distgt} and  \eqref{eq:Contr1} we infer that for large $i$ it holds $\tilde{\bar{x}}_i\in B_{3S\sfd_i}(\tilde{\gamma}(1/2))$.
Therefore by applying the Abresch-Gromoll excess estimate \eqref{eq:AG} we get
\be\label{eq:ExcessAG}
e_{\tilde{x}_i, g_i \tilde{x}_i}(\tilde{\bar{x}}_i) \leq C(K,N) (3S)^{1+\alpha(N)} \sfd_i.
\ee 
On the other hand the definition of excess together with \eqref{eq:ri} and  \eqref{eq:triangletilde} implies
\begin{eqnarray}
e_{\tilde{x}_i, g_i \tilde{x}_i}(\tilde{\bar{x}}_i)&:=&\sfd_{\tilde{X}^{\delta_i}}(\tilde{x}_i, \tilde{\bar{x}}_i)+ \sfd_{\tilde{X}^{\delta_i}}(\tilde{\bar{x}}_i, g_i \tilde{x}_i)-\sfd_{\tilde{X}^{\delta_i}} (\tilde{x}_i, g_i \tilde{x}_i) \nonumber \\
&\geq& 2 \left( \frac{1}{2} +2S \right) \sfd_i -\sfd_i = 4S \sfd_i. \label{eq: LBExc}
\end{eqnarray}
The combination of  \eqref{eq:ExcessAG}  and  \eqref{eq: LBExc} gives
$$4S \sfd_i \leq  C(K,N) (3S)^{1+\alpha(N)} \; \sfd_i \quad \text{for all } i\geq \bar{i}=\bar{i}(S)>>1, $$
which is clearly a contradiction for $S\in (0, \bar{S})$, where  $\bar{S}=\bar{S}(K,N)=\left( \frac{4}{C(K,N) \, 3^{1+\alpha(N)}} \right)^{1/\alpha(N)}$.

\hfill$\Box$

In order to get Theorem \ref{thm:UnCovRCD},  inspired by the work of Sormani and the second author \cite{SW1},  the idea is to use Theorem \ref{thm:LocStab} combined with  \emph{Bishop-Gromov's
relative volume comparison theorem} on $\tilde{X}^{\delta}$ and a  \emph{packing
argument} to show that $\tilde{X}^\delta$ stabilize everywhere. 
\medskip 

\noindent
\emph{Proof of Theorem \ref{thm:UnCovRCD}}.

If the $\delta$-covers $\tilde{X}^\delta$ do not stabilize, then we can find a sequence $\{\delta_i\}_{i \in \N}$ such that 
\be\label{eq:Assdeltai}
\delta_i \downarrow 0, \;   0<\delta_i\leq \diam(X), \; \delta_i\leq 4\, \delta_{i-1}, \; \text{ all $\tilde{X}^{\delta_i}$ and $G(X,\delta_i)$ are distinct.}
\ee
Let us  fix a regular point $x \in X$ and a lifting $\tilde{x}_i\in \pi_i^{-1}(x)\subset \tilde{X}^{\delta_i}$.
The condition that  $G(X,\delta_i)$ are distinct  roughly means  that for every $i \in \N$ we can find $y_i\in X$ such that $B_{\delta_{i-1}}(y_i)$ contains a non-contractible loop based at $y_i$ which is unwrapped by the $\delta_i$-cover $\tilde{X}^{\delta_i}$. More precisely for every $i \in \N$ there exist $y_i \in X$ and $g_i\in G(X,\delta_i)$ such that the following holds:
\begin{enumerate}
\item for every $i\in \N$ fix a lifting $\tilde{y}_i \in \pi_{i}^{-1}(y_i)\subset \tilde{X}^{\delta_i}$ such that
$$
\sfd_{\tilde{X}^{\delta_i}}(\tilde{x}_i, \tilde{y}_i)= \min \left\{\sfd_{\tilde{X}^{\delta_i}}(\tilde{x}_i, \tilde{y}) \,:\,  \tilde{y} \in \pi_{i}^{-1}(y_i)\subset \tilde{X}^{\delta_i} \right\},
$$
in particular
\be\label{eq:txityi}
\sfd_{\tilde{X}^{\delta_i}}(\tilde{x}_i, \tilde{y}_i)=\sfd(x,y_i) \leq \diam(X),
\ee
\item  $ \sfd_{\tilde{X}^{\delta_i}}(\tilde{y}_i, g_i \, \tilde{y}_i)=\min \left\{\sfd_{\tilde{X}^{\delta_i}}(\tilde{y}_i, h \, \tilde{y}_i) \, : \, 1\neq h \in G(X,\delta_i) \right\} $,
\item  called $\tilde{\gamma}_i \subset \tilde{X}^{\delta_i}$ a minimizing geodesic from $\tilde{y}_i$ to $g_i \tilde{y}_i$, then $\gamma_i:=\pi_i(\tilde{\gamma}_i) \subset B_{\delta_{i-1}}(y_i)$ is a closed loop based at $y_i$ representing $g_i$,
\item $\gamma_i \subset B_{\delta_{i-1}}(y_i)$ is half-way minimizing,  in particular
\be\label{eq:gammaiHW}
\frac{1}{2} \sfd_{\tilde{X}^{\delta_i}} (\tilde{y}_i, g_i\, \tilde{y}_i)=  \sfd_{\tilde{X}^{\delta_i}} (\tilde{y}_i, \tilde{\gamma}_i (1/2)) = \sfd (y_i, \gamma_i(1/2)) \leq \delta_{i-1}.
\ee
\end{enumerate}
Since $\tilde{X}^{\delta_j}$ covers $\tilde{X}^{\delta_i}$ for $i=1, \ldots, j-1$, then  there exist  $g_1, \ldots, g_{j-1} \in G(X,\delta_j)$  distinct deck transformations of $\tilde{X}^{\delta_j}$ satisfying the above conditions. For every  $j\in \N$ and $i=1,\dots, j-1$, denote with $\pi^{i}_{j}: \tilde{X}^{\delta_j}\to \tilde{X}^{\delta_i}$ the covering map and let $\tilde{y}_j^i\in (\pi^i_j)^{-1}(\tilde{y}_i) \subset \tilde{X}^{\delta_j}$ be such that
\be\label{eq:distXiXj}
\sfd_{\tilde{X}^{\delta_j}}(\tilde{x}_{j}, \tilde{y}^i_j)= \min \left\{ \sfd_{\tilde{X}^{\delta_j}}(\tilde{x}_{j}, \tilde{y}) \, : \, \tilde{y}\in (\pi^i_j)^{-1}(\tilde{y}_i)\subset \tilde{X}^{\delta_j}\right\}=\sfd_{\tilde{X}^{\delta_i}}(\tilde{x}_i, \tilde{y}_i).
\ee
Notice also that, since $g_i$ is non trivial both in the deck transformations of $\tilde{X}^{\delta_i}$ and of $\tilde{X}^{\delta_j}$, we have 
\be\label{eq:distXiXj1}
\sfd_{\tilde{X}^{\delta_j}}(\tilde{y}^i_{j}, g_i \, \tilde{y}^i_j)=\sfd_{\tilde{X}^{\delta_i}}(\tilde{y}_i, g_i \, \tilde{y}_i).
\ee
Combining then \eqref{eq:txityi},  \eqref{eq:gammaiHW},  \eqref{eq:distXiXj} and \eqref{eq:distXiXj1} we get
\begin{eqnarray}
\sfd_{\tilde{X}^{\delta_j}}(\tilde{x}_j, g_i \, \tilde{x}_j) &\leq& \sfd_{\tilde{X}^{\delta_j}} (\tilde{x}_j, \tilde{y}^{i}_{j}) +  \sfd_{\tilde{X}^{\delta_j}} (\tilde{y}^{i}_{j}, g_i \, \tilde{y}^{i}_{j}) +   \sfd_{\tilde{X}^{\delta_j}} (g_i  \, \tilde{y}^{i}_{j},  g_i \, \tilde{x}_i) \nonumber \\
&=&  \sfd_{\tilde{X}^{\delta_i}} (\tilde{x}_i, \tilde{y}_i) +  \sfd_{\tilde{X}^{\delta_i}} (\tilde{y}_i, g_i \, \tilde{y}_i) +   \sfd_{\tilde{X}^{\delta_j}} (g_i  \, \tilde{y}_i ,  g_i \, \tilde{x}_i) \nonumber \\
&\leq& \diam(X)+ 2 \delta_{i-1} + \diam(X) \leq 4 \diam(X),  \quad \forall i=1,\ldots, j-1, \label{eq:dtxigitxi}  
\end{eqnarray}
where we also used that  $\sfd_{\tilde{X}^{\delta_j}} (g_i  \, \tilde{y}_i ,  g_i \, \tilde{x}_i)= \sfd_{\tilde{X}^{\delta_j}} (\tilde{x}_i, \tilde{y}_i)$, as the deck transformations are isometries.
\\Since by construction $x\in X$ is a regular point, in virtue of Theorem \ref{thm:LocStab} we can find $r_0>0$ such that the ball $B_{r_0}(x)$ is lifted  to  a family of pairwise disjoint balls  $B_{r_0}(h\, \tilde{x}_j) \subset  \tilde{X}^{\delta_j}$, $h\in G(X,\delta_j)$,  each of which is isometric to $B_{r_0}(x)$; moreover \eqref{eq:dtxigitxi}  implies that
\be\label{eq:BroBD}
B_{r_0}(g_i\, \tilde{x}_j) \subset B_{4 \diam(X)+r_0} (\tilde{x}_j)\subset \tilde{X}^{\delta_j},  \quad \forall i=1,\ldots, j-1.
\ee
But now recall that $(\tilde{X}^{\delta_j}, \sfd_{\tilde{X}^{\delta_j}}, \mm_{\tilde{X}^{\delta_j}})$ is an $\RCD^*(K,N)$-space by Lemma \ref{lem:RCDcover} and therefore by Bishop-Gromov relative volume comparison \eqref{eq:BishGrom} we obtain
\be\label{eq:BGXt}
\frac{\mm_{\tilde{X}^{\delta_j}} (B_{r_0}(g_i \, \tilde{x}_j))}{\mm_{\tilde{X}^{\delta_j}} ( B_{4 \diam(X)+r_0} (\tilde{x}_j))} = \frac{\mm_{\tilde{X}^{\delta_j}} (B_{r_0}(\tilde{x}_j))}{\mm_{\tilde{X}^{\delta_j}} ( B_{4 \diam(X)+r_0} (\tilde{x}_j))}  \geq C_{X,r_{0}}>0,   \quad \forall i=1,\ldots, j-1, 
\ee
for some $C_{X,r_{0}}>0$. Since the balls $B_{r_0}(g_i\, \tilde{x}_j)$ are pairwise disjoint and contained in $B_{4 \diam(X)+r_0} (\tilde{x}_j)$, we get
\begin{eqnarray}
j-1 &\leq& \frac{1}{C_{X,r_{0}}}  \; \frac{(j-1)\;  \mm_{\tilde{X}^{\delta_j}} (B_{r_0}(\tilde{x}_j))}{\mm_{\tilde{X}^{\delta_j}} ( B_{4 \diam(X)+r_0} (\tilde{x}_j))} = \frac{1}{C_{X,r_0}}  \frac{\mm_{\tilde{X}^{\delta_j}} ( \bigcup_{i=1}^{j-1} B_{r_0}(g_i \,\tilde{x}_j))}{\mm_{\tilde{X}^{\delta_j}} ( B_{4 \diam(X)+r_0} (\tilde{x}_j))} \nonumber \\
&\leq&  \frac{1}{C_{X,r_0}}  \frac{\mm_{\tilde{X}^{\delta_j}} ( B_{4 \diam(X)+r_0} (\tilde{x}_j))}{\mm_{\tilde{X}^{\delta_j}} ( B_{4 \diam(X)+r_0} (\tilde{x}_j))} =   \frac{1}{C_{X,r_0}} , \nonumber 
\end{eqnarray} 
which gives a contradiction for $j$ large.
\hfill$\Box$

\subsection{Applications to the revised fundamental group of a compact $\RCD^*(K,N)$-space}

The following result was proved in \cite{BS2010} for non-branching $\CD^*(K,N)$-spaces  under the extra assumption of semi-local simply connectedness; recall that the extra assumption of semi-local simply connectedness ensures automatically the existence of a simply connected universal cover, so that the fundamental group $\pi_1(X)$ and the revised fundamental group $\bar{\pi}_1(X)$ coincide. On the other hand both the semi-local simply connectedness and the non-branching assumptions are not stable under $mGH$-convergence, while the $\RCD^{*}(K,N)$ condition is. 

\begin{theorem}\label{thm:finiteRev}
Let $(X,\sfd,\mm)$ be an $\RCD^*(K,N)$-space for some $K>0$ and $N\in (1,\infty)$. Then the revised fundamental group $\bar{\pi}_1(X)$ is finite.
\end{theorem}

\begin{proof}
From Theorem  \ref{thm:UnCovRCD} we know that the universal cover $(\tilde{X},\tilde{\sfd},\tilde{\mm})$ of an $\RCD^*(K,N)$ space $(X,\sfd,\mm)$ exists and is itself $\RCD^*(K,N)$. Recall that the  revised fundamental group $\bar{\pi}_1(X)$ is by definition made by the deck transformations. 
Since the universal cover $\tilde{X}$ is an $\RCD^*(K,N)$-space for $K>0$ then it is has diameter at most $\pi \sqrt{\frac{N}{K}}$, in particular $\tilde{X}$ is compact. If by contradiction we could find infinitely many distinct $g_i \in \bar{\pi}_1(X)$ then, fixed $\tilde{x} \in \tilde{X}$, we would obtain infinitely many isolated points $g_i \tilde{x}\subset \tilde{X}$, contradicting the compactness of $\tilde{X}$. 
\end{proof}

\subsubsection{Extension of Cheeger-Gromoll Theorem to $\RCD^{*}(0,N)$-spaces}
The next result extends to $\RCD^*(0,N)$ spaces a celebrated result by Cheeger-Gromoll  \cite{ChGr} for compact manifolds with non-negative Ricci curvature. The extension seems new even for Ricci-limit spaces. 

\begin{theorem}\label{thm:ChGrpi1}
Let $(X,\sfd,\mm)$ be a compact  $\RCD^{*}(0,N)$-space for some $N \in (1,\infty)$. Then  the revised fundamental group $\bar{\pi}_1(X)$ contains a finite normal subgroup $\psi \lhd \bar{\pi}_1(X)$ such that $ \bar{\pi}_1(X)/\psi$ is a finite group extended by $\underbrace{\Z\oplus \ldots \oplus \Z}_{k \text{ times}}$ and the universal cover $(\tilde{X},\tilde{\sfd}, \tilde{\mm})$ splits isomorphically as m.m.s. as  
$$(\tilde{X}, \tilde{\sfd}, \tilde{\mm}) \simeq (\bar{X}\times \R^{k}, \sfd_{\bar{X}\times \R^{k}}, \mm_{\bar{X}\times \R^{k}}),$$ 
where $(\bar{X}, \sfd_{\bar{X}}, \mm_{\bar{X}})$ is a \emph{compact}   $\RCD^{*}(0,N-k)$-space.
\end{theorem}

\begin{proof}
First of all, by Theorem   \ref{thm:UnCovRCD} we know that if   $(X,\sfd,\mm)$ is a compact  $\RCD^{*}(0,N)$-space for some $N \in (1,\infty)$ then it admits a universal cover   $(\tilde{X}, \sfd_{\tilde{X}}, \mm_{\tilde{X}})$  which is also  an  $\RCD^{*}(0,N)$-space. Moreover, by definition, the group of deck transformations  $G(\tilde{X},X)$ is isomorphic as group to the revised fundamental group  $\bar{\pi}_1(X)$. We prove the result by subsequent steps.
\smallskip

\textbf{Step 1}:    $(\tilde{X},\tilde{\sfd}, \tilde{\mm})$ splits isomorphically as m.m.s. as   $(\tilde{X}, \tilde{\sfd}, \tilde{\mm}) \simeq (\bar{X}\times \R^{k}, \sfd_{\bar{X}\times \R^{k}}, \mm_{\bar{X}\times \R^{k}}),$ 
where $(\bar{X}, \sfd_{\bar{X}}, \mm_{\bar{X}})$ is a \emph{compact}   $\RCD^{*}(0,N-k)$-space, for some $k \in [0,N]\cap \N$. In particular any automorphism $T$ of  $(\tilde{X},\tilde{\sfd}, \tilde{\mm})$ as m.m.s., can be written as a product 
$T=(T_{1}, T_{2})$ where $T_{1}$ is an automophism of $(\bar{X}, \sfd_{\bar{X}}, \mm_{\bar{X}})$ and $T_{2}$ is an automorphism of $(\R^{k}, \sfd_{\R^{k}}, \LL_{k})$ as m.m.s..

If $\tilde{X}$ is compact we can just set $k=0$ and $\bar{X}:=\tilde{X}$; so we can assume $\tilde{X}$ to be non-compact.
\\ Fix a fundamental domain $\tilde{X}_{0}\subset \tilde{X}$, i.e. a  subset of $\tilde{X}$ which is locally isometric and in bijection with $X$ via the projection map $\pi:\tilde{X}\to X$, and a reference point $\tilde{x}_{0}\in \tilde{X}_{0}$. Note that, in particular, $\tilde{X}_{0}$ is pre-compact, as $X$ is compact by assumption.  Since $\tilde{X}$ is non-compact we can find two sequences $(\tilde{p}_{j})_{j\in \N}, (\tilde{q}_{j})_{j\in \N}\subset \tilde{X}$ such that $\tilde{\sfd}(\tilde{p}_{j},\tilde{q}_{j})\to \infty$. Let $\tilde{C}_{j}\subset \tilde{X}$ be a length minimizing geodesic   joining $\tilde{p}_{j}$ with $\tilde{q}_{j}$ and notice that, up to acting on $\tilde{C}_{j}$ with the  isometry induced by a suitable element of the deck transformations   $G(\tilde{X},X)$, we can assume that  the mid-point of $\tilde{C_{j}}$ is contained in $\tilde{X}_{0}$. Since $\tilde{X}$ is proper (recall that, by Theorem   \ref{thm:UnCovRCD},  $(\tilde{X},\tilde{\sfd}, \tilde{\mm})$ is an $\RCD^{*}(0,N)$-space), by using Ascoli-Arzel\'a Theorem we infer that  for every $R>0$ we can find a sub-sequence of $j$'s such that $\tilde{C}_{j}\cap B_{\tilde{x}_{0}}(R)$ converge uniformly  to a length minimizing geodesic $\tilde{C}^{R}_{\infty}\subset  B_{\tilde{x}_{0}}(R)$ such that the midpoint of $\tilde{C}^{R}_{\infty}$ is contained in the closure of $\tilde{X}_{0}$. Considering now a sequence $R_{i}\to \infty$, via a diagonal argument, we then get the existence of a line $\tilde{C}_{\infty}\subset \tilde{X}$, i.e. a curve defined on $(-\infty,\infty)$ which is length-minimizing on every sub-interval $(a,b)$ with $-\infty<a<b<+\infty$.  Applying now the Splitting Theorem \ref{thm:Split} we infer that $(\tilde{X},\tilde{\sfd},\tilde{\mm})$ splits isomorphically as m.m.s. as $(\tilde{X}'\times \R, \sfd_{\tilde{X}'\times \R}, \mm_{\tilde{X}'\times \R})$ where $(\tilde{X}', \sfd_{\tilde{X}'}, \mm_{\tilde{X}'})$ is an $\RCD^{*}(0,N-1)$-space. If now $\tilde{X}'$ is compact we can just set $\bar{X}:=\tilde{X}'$ and $k=1$. 

 If instead   $\tilde{X}'$  is not compact we can argue analogously as above:  let $(\tilde{p}_{j}')_{j\in \N}, (\tilde{q}_{j}')_{j\in \N}\subset \tilde{X}'$ such that $\sfd_{\tilde{X}'}(\tilde{p}_{j}',\tilde{q}_{j}')\to \infty$ and let  $\tilde{C}_{j}'\subset \tilde{X}'$ be a length minimizing geodesic joining them; up to acting with the isometries induced by deck transfomations   $G(\tilde{X},X)$ and with the isometries induced by the translations along the $\R$ factor in the isometric splitting $\tilde{X}=\tilde{X}'\times \R$, we can assume that $(\tilde{C}_{j}'(1/2), t_{0}) \in \tilde{X}_{0}'$, for some fixed $t_{0}\in \R$. In particular  the mid-points of $\tilde{C}'_{j}$ are contained in a  fixed compact subset in $\tilde{X}'$ and we can then repeat the above diagonal argument producing a line in $\tilde{X}'$. As above, in virtue of the  Splitting Theorem \ref{thm:Split} we infer that $(\tilde{X},\tilde{\sfd},\tilde{\mm})$ splits isomorphically as m.m.s. as $(\tilde{X}''\times \R^{2}, \sfd_{\tilde{X}''\times \R^{2}}, \mm_{\tilde{X}''\times \R^{2}})$ where $(\tilde{X}'', \sfd_{\tilde{X}''}, \mm_{\tilde{X}''})$ is an $\RCD^{*}(0,N-2)$-space. If now $\tilde{X}''$ is compact we can just set $\bar{X}:=\tilde{X}'$ and $k=2$, otherwise we repeat the procedure above to $\tilde{X}''$. Since the Hausdorff dimension of an $\RCD^{*}(0,N)$ space is at most $N$, the iteration can be repeated at most $N$ times;  the claim is thus proved.
 
 The splitting of the group of automorphisms follows by the splitting of the group of isometries which was proved in the smooth setting by Cheeger-Gromoll \cite{ChGr}; for a proof in the non-smooth setting see for instance \cite[Proposition 1]{ShWei}.
 \medskip

\textbf{Step 2}:   Called $\pi_{\bar{X}}: \tilde{X} \to \bar{X}$ the projection on the first coordinate with respect to the splitting $\tilde{X}\simeq \bar{X}\times \R^{k}$, we claim that $(\pi_{\bar{X}} )_{*}(G(\tilde{X},X))$ is finite.

We first claim that the projection of the fundamental domain $\tilde{X}_{0}$ has positive measure in $\bar{X}$, namely 
\begin{equation}\label{eq:mpi>0}
\mm_{\bar{X}}(\pi_{\bar{X}}(\tilde{X}_{0}))>0.
\end{equation}
Indeed if by contradiction it holds $\mm_{\bar{X}}(\pi_{\bar{X}}(\tilde{X}_{0}))=0$, since from Step 1 we know that  $\tilde{X}_{0}\subset \pi_{\bar{X}}(\tilde{X}_{0})\times \R\subset \tilde{X}$ as m.m.s., it would follow that $\tilde{\mm}(\tilde{X}_{0})=0$. But since  $\tilde{X}_{0}\subset \tilde{X}$ is a fundamental domain  this would imply $\mm(X)=\tilde{\mm}(\tilde{X}_{0})=0$, contradicting that $\mm(X)>0$.  This concludes the proof of \eqref{eq:mpi>0}.

Now by Step 1 we know that $\bar{X}$ is a compact $\RCD^{*}(0,N-k)$ space, in particular it has finite volume. Since  the elements of $(\pi_{\bar{X}} )_{*}(G(\tilde{X},X))$  act on $\bar{X}$ by isomorphisms of m.m.s., if by contradiction  $(\pi_{\bar{X}} )_{*}(G(\tilde{X},X))$ was infinite it would follow that $\bar{X}$ contains infinitely many isomorphic copies of $\pi_{\bar{X}}(\tilde{X}_{0})$ and therefore \eqref{eq:mpi>0} would force  $\bar{X}$ to have infinite volume; contradiction.

 \medskip

\textbf{Step 3}: Called $\pi_{\R^{k}}: \tilde{X} \to \R^{k}$ the projection on the second coordinate with respect to the splitting $\tilde{X}\simeq \bar{X}\times \R^{k}$ and  denoted by $G_{\R^{k}}:=(\pi_{\R^{k}} )_{*}(G(\tilde{X},X))$, 
we claim that $G_{\R^{k}}$ contains a normal rank $k$ free abelian subgroup  of finite  index.

By the product structure of $\tilde{X}$, it follows that  the action of  $G_{\R^{k}}:=(\pi_{\R^{k}} )_{*}(G(\tilde{X},X))$ on  $\R^{k}$ is by uniform and discrete isometries. Bieberbach Theorem (see for instance \cite[Section 1.1]{Dekimpe} and references therein) implies then that  $G_{\R^{k}}$ contains a normal  free abelian subgroup  of finite  index and of rank  $l\leq k$.
 Moreover, from the construction above we have that 
\begin{equation}
\bigslant{\R^{k}}{G_{\R^{k}}}  \simeq \pi_{\R^{k}} (\tilde{X}_{0}) \text{ isometric as m.m.s.}. 
\end{equation}
Since by Step 1 we know that  $\tilde{X}_{0}$ is compact, it follows that $\bigslant{\R^{k}}{G_{\R^{k}}}$ is compact too and therefore $l=k$ as desired.   \medskip 

\textbf{Step 4}: Conclusion of the proof.

From Step 3 we know that $G_{\R^{k}}:=(\pi_{\R^{k}} )_{*}(G(\tilde{X},X))$  is a finite group extended by $\Z^{k}$.  The projection map $\pi_{\R^{k}}:\tilde{X}\to \R^{k}$ induces the short exact sequence of groups
$$
0\to \text{Ker}\left( (\pi_{\R^{k}})_{*} \right) \to G(\tilde{X},X)\to G_{\R^{k}}\to 0,
$$
which in turn implies that $G(\tilde{X},X)/  \text{Ker}\left( (\pi_{\R^{k}})_{*} \right) \simeq  G_{\R^{k}}$. Since  $\tilde{X}$ is an isometric splitting $\bar{X}\times \R^{k}$ and since the deck transformations $G(\tilde{X},X)$ are isometries which preserve such a splitting structure, it follows that $\text{Ker}\left( (\pi_{\R^{k}})_{*} \right) \simeq (\pi_{\bar{X}} )_{*}(G(\tilde{X},X))$ which  is finite by Step 2.  Denoting $\psi:=\text{Ker}\left( (\pi_{\R^{k}})_{*} \right),$ the thesis follows.

\end{proof}

\begin{corollary}\label{cor:Rigid}
Let $(X,\sfd,\mm)$ be a compact  $\RCD^{*}(0,N)$-space for some $N \in (1,\infty)$. If the revised fundamental group $\bar{\pi}_{1}(X)$ contains $\lfloor N \rfloor$ independent generators of infinite order, then $(X,\sfd,\mm)$ is isomorphic as m.m.s. to an $N$-dimensional flat Riemannian manifold, i.e. there exists a covering space $(\hat{X},\hat{\sfd},\hat{\mm})$ of $(X,\sfd,\mm)$ which is isomorphic as m.m.s. to a flat torus ${\mathbb T}^{N}=\R^{N}/\Gamma$, for some lattice $\Gamma\subset \R^{N}$.
\end{corollary}

\begin{proof}
If $\bar{\pi}_{1}(X)$ contains  $\lfloor N \rfloor$  independent generators of infinite order, then Step 1 in the proof of Theorem \ref{thm:ChGrpi1} can be repeated  $\lfloor N \rfloor$  times, forcing the universal cover $\tilde{X}$ to be isomorphic to $\R^{ \lfloor N \rfloor }$ as m.m.s. since the cross section would be an $\RCD^{*}(N- \lfloor N \rfloor, 0)$ space and therefore a point.  The thesis follows.
\end{proof}

\section{Non-compact $\RCD^*(K,N)$-spaces admit a universal cover}

The goal of the present section is to prove the following result. The proof will be in the same spirit of the arguments of Theorem  \ref{thm:UnCovRCD}, but slightly more complicated since here we have to  deal with relative $\delta$-covers instead of simply $\delta$-covers. This is indeed the main reason why we decided to present both proofs, so that the  reader interested to get the geometric ideas can just read the proof of Theorem  \ref{thm:UnCovRCD}.

\begin{theorem}\label{thm:UnCovRCD-noncompact}
	Let $(X,\sfd,  \mm)$ be an  $\RCD^*(K,N)$-space for some $K\in \R$, $N\in(1, +\infty)$.  Then $(X,\sfd,  \mm)$ admits a universal cover $(\tilde{X},\tilde{\sfd},\tilde{\mm})$ which is itself an  $\RCD^*(K,N)$-space. 
\end{theorem}

\begin{remark}\label{rem:SimplicationNonCompact}
In \cite{SW2}, the analogous result is shown for Gromov-Hausdorff limits of manifolds $M^n$ with Ricci curvature uniformly bounded from below. The proof in \cite{SW2} is much more involved as the authors had to construct a measure for the limit $\delta$-cover; indeed  it was not clear the existence of  a measure satisfying Bishop-Gromov comparison  on the relative $\delta$-covers of the limit spaces, property which is used in a crucial way in the proof.  
\end{remark}

\begin{remark}When $(X,\sfd)$ is semi-locally simply connected, the universal cover of $X$ can be obtained as Gromov-Hausdorff limit of universal cover of larger and larger balls, see \cite[Proposition 1.2]{Ennis-Wei2010}. Since here we do not want to assume this extra hypothesis   we will argue differently.
\end{remark}
We need the following local stability of relative $\delta$-cover at regular points, which is the corresponding result  of \cite[Theorem 3.13]{SW2} proved there for limit spaces, compare also with Theorem~\ref{thm:LocStab}. 

\begin{theorem}\label{thm:LocStab-noncompact}
	Let $(X,\sfd, \mm)$ be an  $\RCD^*(K,N)$ space for some $K\in \R, N\in (1,\infty)$ and let $x \in X$ be a  regular point. Then for any 
	$100 \le 10 \bar r \le \bar R$ there
	exists $r_x(\bar r,\bar R)>0$, such that for
	all $\delta>0$, $B(x,r_x)$ lifts isometrically
	to $\tilde{B}(x,\bar r,\bar R)^\delta$.
\end{theorem}

The proof of Theorem~\ref{thm:LocStab-noncompact}
uses the Abresch-Gromoll Excess estimate \eqref{eq:AG} on the relative
$\delta$ cover as in Theorem~\ref{thm:LocStab}
except that now our covers have boundary. By choosing $r_x \le 1$, all points and curves involved in the proof of Theorem~\ref{thm:LocStab} lie in $B(x, 4)$, far away from the  boundary $\partial B(x, \bar{r})$, and similarly for the cover. By Lemma~\ref{distance}, the restricted distance on $B(x, 4)$ from
$B(x,\bar r)$ and from $X$ are the  same. So the proof of
Theorem~\ref{thm:LocStab}  can be repeated verbatim. 

Now we prove stability of relative $\delta$-covers for all points. 
\begin{theorem} \label{samedelta2}
	Let $(X,\sfd, \mm)$ be an  $\RCD^*(K,N)$ space for some $K\in \R, N \in (1,\infty)$.	For all $R>0$ and $x \in X$ one of the following two statements
	holds true:
	
	I: There exists $\delta_{x,R}$ depending on
	$X, x, R$ such that 
	\be
	\tilde{B}(x,R)^\delta=\tilde{B}(x,R)^{\delta_{x,R}} \qquad \forall \delta< \delta_{x,R}.
	\ee
	
	II. For all $R'<R$ there exists $\delta_{x,R,R'}$ depending
	on $X,x,R,R'$ such that
	\be
	\tilde{B}(x,R',R)^{\delta_{x,R,R'}}=\tilde{B}(x,R',R)^\delta
	\qquad \forall \delta< \delta_{x,R,R'}.
	\ee
\end{theorem}
The proof should be compared  with \cite[Theorem 3.15]{SW2}, the main difference is that  here we do not need to go to the limit cover; this is a  quite useful  simplification in the arguments. 
\begin{proof}
Suppose neither I nor II holds.
Then there exists $x$, $R$ and $\delta_i\downarrow 0$ such
that $\tilde{B}(x,R)^{\delta_i}$ are all distinct.  Up to extracting a subsequence, we can assume that
$\delta_1 \leq R/10, \delta_i > 10 \delta_{i+1}$
and  all $\tilde{B}(x,R)^{\delta_i}$  and $G(x,R,\delta_i)$
are distinct.   In particular there are non-trivial elements of
$G(x,R, \delta_i)$ which are trivial in $G(x,R, \delta_{i-1})$.
Thus, for every $i \in \N$, we can find  $x_i\in B(x,R)$ such that  $B_{x_i}(\delta_{i-1})$ contains
a non-contractible loop, $C_i$, which lifts non-trivially in
$\tilde B(x,R)^{\delta_i}$.

In fact we can choose $x_1$ to be the point
closest to $x$ such that $B_{x_1}(R/10)$ contains a non-contractible
loop and then choose $\delta_1\in (0,R/10]$ as small as possible
such that $B_{x_1}(R/10)$ contains a loop $C_1$ which lifts non-trivially to
$\tilde B(x,R)^{\delta_1}$.  We can then choose iteratively
$x_j$  to be the point closest to $x$ such that $B_{x_j}(\delta_{j-1}/10)$
contains a non-contractible loop.  Then set $\delta_{j}\in (0,\delta_{j-1}/10]$
as small as possible so that $B_{x_j}(\delta_{j})$ contains a
loop $C_j$ which lifts non-trivially to $\tilde B(x,R)^{\delta_j}$.  Note that by construction 
$\sfd_{B(x,R)}(x,x_j)$ is a non-decreasing sequence.

Since $\RCD^{*}(K,N)$ spaces  are proper for $N<\infty$,  we can find  a sub-sequence of the $x_i$ converging to some point $x_\infty$ in $B(x,R)$. There are two possibilities: either  $x_\infty \in \partial B(x,R)$ or $x_{\infty}\in B_{x}(R)$.

If $x_\infty\in \partial B(x,R)$, then for any $0<R'<R$, we know that
there exists $N_1$ sufficiently large such that
$\sfd_{B(x,R)}(x,x_j)>R'+(R-R')/2$ for all $j\ge N_1$. Moreover it is clear that  $\delta_{(j-1)/10} <(R-R')/2$ for
all $j\ge N_2$, for some  $N_2\geq N_1$. 
\\By the choice of the sequence of $x_j$,
it follows  that if $C$ is a loop contained in $B(x',\delta)$
where $\delta\le \delta_{(N_2-1)/10}$ and
$B(x',\delta)\cap B(x,R')$ is nonempty, then $C$ is contractible
in $B(x,R)$.
Thus
\be
\tilde{B}(x,R',R)^{\delta_{x,R,R'}}=\tilde{B}(x,R',R)^\delta  \ \
\forall \delta< \delta_{(N_2-1)/10} =: \delta_{x,R,R'}.
\ee
This implies Case II which we have assumed to be false. Therefore $x_\infty$ is not in the boundary of $B(x,R)$, and we proceed
to find a contradiction.

Let
$\bar{R}>0$ be defined such that $B(x_\infty, \bar{R})\subset B(x,R)$
and let $\bar{r}=\bar{R}/10>0$.  Clearly, up to throwing away finitely many  $j$, we can assume that   the loops $C_j$ are contained 
in $B(x_\infty, \bar{r}/6)$ for every $j \in \N$.  Since by construction $C_j$ are contained in $B_{x_j}(\delta_{j-1})$,   they lift as closed curves to
$\tilde{B}(x_{\infty},\bar{r},\bar{R})^{\delta_{i}}$ for every $i<j$.  On the other hand, as they lift non-trivially to
$\tilde B(x,R)^{\delta_{j}}$, they also lift non-trivially to
$\tilde{B}(x_\infty,\bar{r},\bar{R})^{\delta_{j}}$.

Since without loss of generality we can assume $C_j$ to be half-way minimizing, 
there exists $g_j$ non-trivial in
$G(x_\infty,\bar{r},\bar{R},\delta_j)$ such that
\be
\sfd_{\tilde{B}(x_\infty,\bar{r},\bar{R})^{\delta_j}} (g_j\tilde{x}_j, \tilde{x}_j) < 2 \delta_{j-1}.
\ee
Let $\alpha_{j}$ be the projection of the minimal curve from
$g_j\tilde{x}_j$ to $\tilde{x}_j$.  Then $L(\alpha_j)<2\delta_{j-1}<\bar{r}/6$
and
\be
\alpha_j\subset B_{x_j}(2\delta_{j-1})\subset B(x_\infty, \bar{r}/6+2\delta_{j-1})
\subset B(x_\infty, \bar{r}/3).
\ee
The loop $\alpha_j$ represents an element $g_j$
of ${\pi}_1(B(x_\infty,\bar{r}))$ which is mapped non-trivially into
$G(x_\infty,\bar{r},\bar{R},\delta_j)$ and
trivially into $G(x_\infty,\bar{r},\bar{R},\delta_{i})$ for every $i<j$.
Furthermore, for any $q\in B(x_\infty,\bar{r})$, letting $\tilde{q}\in \tilde{B}(x_\infty,\bar{r},\bar{R})^{\delta_j}$
be the lift of $q$ closest  to $\tilde{x}_j$ and letting $\tilde{x}_{i} \in \tilde{B}(x_\infty,\bar{r},\bar{R})^{\delta_j}$ be the lift of $x_{i}$ closest to $\tilde{q}$,
we have 
\begin{eqnarray*}
	\sfd_{\tilde{B}(x_\infty,\bar{r},\bar{R})^{\delta_j}}(\tilde{q}, g_i\tilde{q})
	& \le & \sfd_{\tilde{B}(x_\infty,\bar{r},\bar{R})^{\delta_j}}(\tilde{q},
	\tilde{x}_i)
	+\sfd_{\tilde{B}(x_\infty,\bar{r},\bar{R})^{\delta_j}}(\tilde{x}_i,
	g_i\tilde{x}_i)
	+\sfd_{\tilde{B}(x_\infty,\bar{r},\bar{R})^{\delta_j}}(g_i\tilde{x}_i,
	g_i\tilde{q}) \\ &\le& 2 \bar{r} +  L(\alpha_i) + 2\bar{r} \le
	5\bar{r}. \end{eqnarray*}

Therefore, for any $j\in \N$,
there are $j-1$  distinct elements in
$G(x_\infty,\bar{r},\bar{R})^{\delta_j}$ satisfying 
\begin{equation}\label{eq:lgidj}
l(g_i, \delta_j) := \inf_{q\in B(x_\infty,\bar{r})}
\sfd_{\tilde{B}(x_\infty,\bar{r},\bar{R})^{\delta_j}}(\tilde{q}, g_i\tilde{q}) \leq 5\bar{r}.
\end{equation}

On the other hand we claim that the total number of elements in
$G(x_\infty, \bar{r}, \bar{R})^\delta$ of
$\delta$-length $\leq 5\bar{r}$ is uniformly bounded for all $\delta$ in
terms of the  geometry  of $B(x_\infty,\bar{r})$.
\\To show this claim, let us look
at the lift of a regular point $p \in B(x_\infty,\bar{r})$ in the cover
$B(x_\infty,\bar{r},\bar{R})^{\delta/2}$.
In virtue of Theorem~\ref{thm:LocStab-noncompact}, we can find  a $\delta_p>0$ such that
the ball of radius $\delta_p$ about $p$ is isometrically lifted
to disjoint balls of radius $\delta_p$ in $\tilde{B}(x_\infty,\bar{r},\bar{R})^\delta$.
Let
\be
\delta_0=\min\{\delta_p, \bar{r}\}
\ee
and denote with  $N\geq j-1$ the number of distinct elements in
$G(x_\infty, \bar{r}, \bar{R})^\delta$ of
$\delta$-length $\leq 5\bar{r}$. Note that $g B(\tilde{p}, \delta_0)$
is contained in
$B(\tilde{p}, 5\bar{r}+\delta_0 )\subset
\tilde{B}(x_\infty,\bar{r},\bar{R})^\delta$ for all
$g \in G(x_\infty, \bar{r}, \bar{R})^\delta$
with $l(g, \delta) \leq 5\bar{r}$. Hence there are $N+1$ isometric disjoint balls of
radius $\delta_0$ contained in a ball of radius
$5\bar{r}+\delta_0$ in $\tilde{B}(x_\infty,\bar{r},\bar{R})^\delta$.  Here we have included the
center ball as well.
\\Since  $5\bar{r}+\delta_0<\bar{R}-\bar{r}$, the ball $B(\tilde{q}, 5\bar{r}+\delta_0)$ does not touch the boundary of $\tilde{B}(x_\infty,\bar{r},\bar{R})^\delta$. Therefore  \eqref{eq:BishGrom} holds  on  $B(\tilde{q}, 5\bar{r}+\delta_0)$ and we get
\be
j\leq N+1 \leq  \frac{\mm (B(\tilde{q}, 5\bar{r}+\delta_0))}
{\mm (B(\tilde{q}, \delta_0))}
\le \frac{1}{\Lambda_{K,N} (\delta_0,   5\bar{r}+\delta_0)},
\ee
which gives us a contradiction for $j$ large enough.	
\end{proof}
\medskip

\textbf{Proof of Theorem \ref{thm:UnCovRCD-noncompact}}
The existence of a universal  cover $\tilde{X}$ follows by the combination of Theorem \ref{univstab} with  Theorem \ref{samedelta2} (note that if Case I holds in Theorem \ref{samedelta2} we can just take $r=R$ in Theorem \ref{univstab}).
Using the construction of Section  \ref{ss:LiftMMS} we can lift the metric $\sfd$ and the measure $\mm$ to a metric $\tilde{\sfd}$ and a  measure $\tilde{\mm}$ on $\tilde{X}$, so that the universal cover is a metric measure space $(\tilde{X}, \tilde{\sfd}, \tilde{\mm})$.
We now claim that $(\tilde{X}, \tilde{\sfd}, \tilde{\mm})$ is an $\RCD^{*}(K,N)$ space. 
To prove such claim we observe that the universal cover $(\tilde{X},\tilde{\sfd})$ is constructed as covering space $\tilde{X}_{\mathcal U}$ associated to an open cover ${\mathcal U}$ of $X$ with the property that for every $x\in X$ there exists $U_{x} \in \mathcal{U}$ such  that $U_{x}$ is lifted homeomorphically to any covering space of $(X,\sfd)$, see \cite[Lemma 2.4, Theorem 2.5]{SW2}.  The claim is then a consequence of Lemma \ref{lem:RCDcover}.
\hfill$\Box$

\subsection{Applications to the revised fundamental group of a non-compact $\RCD^*(0,N)$-space}

\subsubsection{Extension of Milnor's Theorem  to $\RCD^{*}(0,N)$-spaces}
We can extend to $\RCD^*(0,N)$-spaces Milnor's result   \cite{Mil} about the polynomial growth of finitely generated subgroups of the fundamental group of a non-compact manifold with non-negative Ricci curvature, for the  extension to Ricci-limits see \cite{SW2}. 

The idea is to use the polynomial  volume growth in the universal cover ensured by the curvature condition  in order to get information on the growth of the revised fundamental group. Let us start with some preliminary notions.

Let $G$ be a finitely generated group: $G=\langle g_{1},\ldots, g_{n}\rangle$. We define the $r$-neighborhood with respect to the set of generators ${\mathfrak g}:=\{g_{1},\ldots,g_{n}\}$ as
\begin{equation}\label{eq:defUgr}
U_{{\mathfrak g}}(r):=\{g \in G \,:\, g=g^{i_{1}}_{\alpha_{1}}\cdot g^{i_{2}}_{\alpha_{2}}\cdot \ldots \cdot g^{i_{k}}_{\alpha_{k}}, \; \text{with }  \alpha_{j}\in\{1,\ldots,n\},  \text{for }j=1,\ldots,k, \text{ and } \; \sum_{j=1}^{k} |i_{j}|\leq r  \}, 
\end{equation}
i.e. $U_{{\mathfrak g}}(r)$ is made of words of length at most $r$ with respect to the generating family ${\mathfrak g}$.

\begin{definition}
The group $G$ is said to have \emph{polynomial growth of degree $s\in (1,\infty)$} if there exist a set of generators ${\mathfrak g}$, and real numbers $C, r_{0}\geq 1$  such that $|U_{{\mathfrak g}}(r)|\leq C r^{s}$ for every $r\geq r_{0}$.
\end{definition}

\begin{theorem}\label{thm:Milnor}
Let $(X,\sfd, \mm)$ be an $\RCD^*(0,N)$-space for some $N\in(1,\infty)$. Then any finitely generated subgroup of the revised fundamental group $\bar{\pi}_1(X)$ has polynomial growth of degree at most $N$.
\end{theorem}

\begin{proof}
Thanks to Theorem \ref{thm:UnCovRCD-noncompact}, we know that  $(X,\sfd, \mm)$ admits a universal cover  $(\tilde{X}, \tilde{\sfd}, \tilde{\mm})$ which is an $\RCD^{*}(0,N)$-space too. Moreover by the very definition, the revised fundamental group $\bar{\pi}_{1}(X)$ is isomorphic to the group of deck transformations $G(\tilde{X},X)$. Let $G_{0}:= \langle g_{1},\ldots, g_{n}\rangle$ be a finitely generated subgroup of  $G(\tilde{X},X)$.  Fix a reference point $x_{0}\in X$ and fix a lift $\tilde{x}_{0}\in \tilde{X}$, i.e. $\pi(\tilde{x}_{0})=x_{0}$. Each $g_{i}$ can be represented by a loop $C_{i}$ based at $x_{0}$ of length $L_{i}$; then the lifts $\tilde{C}_{i}$ starting at $\tilde{x}_{0}$ are curves in $\tilde{X}$  with final point $\tilde{C}_{i}(L_{i})=g_{i}(\tilde{x}_{0})$. Let
$$\varepsilon:= \frac{1}{3} \min \{L_{1}, \cdots, L_{n} \}, \quad L:= \max \{L_{1}, \cdots, L_{n} \}.$$
Then, for every distinct $g, g' \in G_{0}$ it holds $g(B_{\tilde{x}_{0}}(\varepsilon))\cap g'(B_{\tilde{x}_{0}}(\varepsilon))=\emptyset$; moreover $\bigcup_{g\in U_{{\mathfrak g}}(r)} g(B_{\tilde{x}_{0}} (\varepsilon)) \subset B_{\tilde{x}_{0}}(rL+\varepsilon)$, where of course ${\mathfrak g}:=\{g_{1},\ldots,g_{n}\}$. It follows that
$$
|U_{{\mathfrak g}}(r)| \cdot \tilde{\mm}(B_{\tilde{x}_{0}} (\varepsilon))=\sum_{g\in U_{{\mathfrak g}}(r)}  \tilde{\mm}\big(g(B_{\tilde{x}_{0}}(\varepsilon))\big) \leq \tilde{\mm}  \big(B_{\tilde{x}_{0}}(rL+\varepsilon)\big),
$$
and therefore, by using \eqref{eq:BGK0}, we infer
$$
|U_{{\mathfrak g}}(r)|\leq \frac{\tilde{\mm}  \big(B_{\tilde{x}_{0}}(rL+\varepsilon)\big)}{ \tilde{\mm}(B_{\tilde{x}_{0}} (\varepsilon))} \leq \left(\frac{rL+\varepsilon}{\varepsilon} \right)^{N} \leq C_{\varepsilon,L,N} r^{N}, \quad \text{for } r\geq r_{0}:=1.
$$

\end{proof}

One of the most striking results about groups of polynomial growth is the following theorem by Gromov.

\begin{theorem}[Gromov \cite{Gromov81b}]\label{thm:Gromov}
A group has polynomial growth if and only if
it is almost nilpotent,  i.e.  it contains a nilpotent subgroup of finite index.
\end{theorem}

Combining Theorem \ref{thm:Milnor} with Theorem \ref{thm:Gromov}, we get the following corollary.

\begin{corollary}
Let $(X,\sfd, \mm)$ be an $\RCD^*(0,N)$-space for some $N\in(1,\infty)$. Then any finitely generated subgroup of the revised fundamental group $\bar{\pi}_1(X)$ is almost nilpotent,  i.e.  it contains a nilpotent subgroup of finite index.
\end{corollary}

\subsubsection{Extension of Anderson's Theorem to $\RCD^{*}(0,N)$-spaces}
We can also extend to $\RCD^*(0,N)$-spaces Anderson's Theorem \cite{An} about maximal volume growth and finiteness of the  revised
fundamental group,  for the  extension to Ricci-limits see \cite{SW2}.

\begin{theorem}\label{thm:Anderson}
Let $(X,\sfd,\mm)$ be an $\RCD^*(0,N)$-space for some $N \in (1,\infty)$ and assume that it has euclidean volume growth, i.e. $\liminf_{r\to +\infty} \mm(B_{x_{0}}(r))/r^N=C_{X}>0$. Then the revised fundamental group is finite. More precisely, it holds  $$|\bar{\pi}_1(X,x_0)|\leq \tilde{\mm}(B_{1}(\tilde{x}))/C_{X},$$
where $(\tilde{X}, \tilde{\sfd}, \tilde{\mm})$ is the universal cover of $(X,\sfd,\mm)$ and $\tilde{x}\in \tilde{X}$ is a fixed reference point.
 \end{theorem}
 
 \begin{proof}
 Thanks to Theorem \ref{thm:UnCovRCD-noncompact}, we know that  $(X,\sfd, \mm)$ admits a universal cover  $(\tilde{X}, \tilde{\sfd}, \tilde{\mm})$ which is an $\RCD^{*}(0,N)$-space too. Moreover by the very definition, the revised fundamental group $\bar{\pi}_{1}(X)$ is isomorphic to the group of deck transformations $G(\tilde{X},X)$. Let $\hat{G}=\langle g_{1}, \ldots, g_{n}\rangle$ be a finitely generated subgroup of  $G(\tilde{X},X)$,  set ${\mathfrak g}:=\{g_{1}, \ldots, g_{n}\}$ the fixed system of generators and let $(\hat{X},\hat{\sfd},\hat{\mm})$ be the covering space of $X$ such that $\hat{\pi}:\tilde{X}\to \hat{X}$ has covering group $\hat{G}$; in other words set $\hat{X}=\tilde{X}/\hat{G}$. The goal is to show that there is a uniform upper bound on $|\hat{G}|$.
\\

\textbf{Step 1}. $\hat{G}$ is finite.
 \\Fix a reference point $x_{0}\in X$ and fix lifts $\tilde{x}_{0}\in \tilde{X}$, $\hat{x}_{0}\in \hat{X}$. Let $\tilde{X}_{\hat{G}}\subset \tilde{X}$ be a fundamental domain for the action of $\hat{G}$ such that $\tilde{x}_{0}\in \tilde{X}_{\hat{G}}$; since by construction $\hat{\pi}:\tilde{X}\to \hat{X}$ is locally an isomorphism of m.m.s. and is injective on $\tilde{X}_{\hat{G}}$, it follows that 
\begin{equation} \label{eq:hmtm}
 \hat{\mm}(B_{\hat{x}_{0}}(r))=\tilde{\mm}(B_{\tilde{x}_{0}}(r)\cap \tilde{X}_{\hat{G}}) , \quad \text{ for all $r\geq 0$}.
 \end{equation}
 Let $U_{{\mathfrak g}}(r)$ be the  $r$-neighborhood with respect to the set of generators ${\mathfrak g}:=\{g_{1},\ldots,g_{n}\}$ defined in \eqref{eq:defUgr}. Each $g_{i}$ can be represented by a loop $C_{i}$ based at $x_{0}$ of length $L_{i}$; then the lifts $\tilde{C}_{i}$ starting at $\tilde{x}_{0}$ are curves in $\tilde{X}$  with final point $\tilde{C}_{i}(L_{i})=g_{i}(\tilde{x}_{0})$. Let $L:= \max \{L_{1}, \cdots, L_{n} \}.$
Then, 
$$\bigcup_{g\in U_{{\mathfrak g}}(r)} g(B_{\tilde{x}_{0}} (r)) \subset B_{\tilde{x}_{0}}(r(L+1)), \quad \forall r>0.$$
It follows that
$$
|U_{{\mathfrak g}}(r)| \cdot \tilde{\mm}(B_{\tilde{x}_{0}} (r))\leq \sum_{g\in U_{{\mathfrak g}}(r)}  \tilde{\mm}\big(g(B_{\tilde{x}_{0}}(r))\big) \leq \tilde{\mm}  \big(B_{\tilde{x}_{0}}((r+1)L)\big), \quad \forall r>0.
$$
Since   $(\tilde{X}, \tilde{\sfd}, \tilde{\mm})$ is an $\RCD^{*}(0,N)$-space, Bishop-Gromov inequality \eqref{eq:BGK0} implies that,  called $C_{\tilde{X}}:=\tilde{\mm}(B_{\tilde{x}_{0}}(1))>0$, it holds 
 \begin{equation}\label{eq:Poltmm}
 \tilde{\mm}(B_{\tilde{x}_{0}}(r))\leq  C_{\tilde{X}} r^{N}, \quad \text{ for all $r\geq 1$}. 
 \end{equation}
On the other hand,   by assumption, for large $r$ we have  $\mm(B_{x_{0}}(r))\geq C_{X} r^{N}$. Combining  the last three informations we get
$$
\limsup_{r\to \infty} |U_{{\mathfrak g}}(r)| \leq \limsup_{r\to \infty} \frac{ \tilde{\mm}  \big(B_{\tilde{x}_{0}}((r+1)L)\big) }{\tilde{\mm}(B_{\tilde{x}_{0}} (r))} \leq   \limsup_{r\to \infty} \frac{ C_{\tilde{X}} \big((r+1) L\big)^{N}} { C_{X} r^{N}} = \frac{C_{\tilde{X}} }{ C_{X}}<\infty,
$$
which proves that  $\hat{G}$ is finite.
\\ 
 
 \textbf{Step 2}. $|\hat{G}| \leq  \tilde{\mm}(B_{1}(\tilde{x}))/C_{X}.$ 
 \\Since by Step  1 we know that $\hat{G}$ is finite,   there exists $R=R_{\hat{G}}>0$ such that 
 \begin{equation}\label{eq:BgrR}
 B_{g\tilde{x}_{0}}(r)\subset B_{\tilde{x}_{0}}(r+R), \quad \text{ for all $r>0$ and for all $g \in \hat{G}$}.
 \end{equation}
 Putting together  \eqref{eq:hmtm},  \eqref{eq:Poltmm} and \eqref{eq:BgrR}  we then infer
 \begin{eqnarray}
 |\hat{G}| \; \hat{\mm}(B_{\hat{x}_{0}}(r))&=&   |\hat{G}| \;  \tilde{\mm} \big(B_{\tilde{x}_{0}}(r) \cap \tilde{X}_{\hat{G}} \big) = \sum_{g\in \hat{G}} \tilde{\mm} \big(B_{g\tilde{x}_{0}}(r) \cap g\cdot \tilde{X}_{\hat{G}} \big) \   \nonumber \\
&\leq&  \sum_{g \in \hat{G}}  \tilde{\mm} \big(B_{\tilde{x}_{0}}(R+r) \cap g\cdot \tilde{X}_{\hat{G}} \big)  \leq  \tilde{\mm} \big(B_{\tilde{x}_{0}}(R+r) \big)  \leq  C_{\tilde{X}} (r+R)^{N}, \quad \text{ for all $r\geq 1$}.   \nonumber
 \end{eqnarray}
Recall that,  by assumption, for large $r$ we have  $\mm(B_{x_{0}}(r))\geq C_{X} r^{N}$. Since $\hat{X}$ is a covering space of $X$, a fortiori it must hold that $\hat{\mm}(B_{\hat{x}_{0}}(r))\geq C_{X} r^{N}$.  Thus
$$|\hat{G}| \, C_{X} r^{N}  \leq  C_{\tilde{X}} (r+R)^{N}, \quad \text{for  $r>1$ large enough},$$
 which yields
 $$|\hat{G}|\leq \frac{C_{\tilde{X}}}{C_{X}} \lim_{r\to \infty} \left(\frac{r+R}{r} \right)^{N}= \frac{C_{\tilde{X}}}{C_{X}}.$$
 Therefore, there is a uniform bound (depending just on $X$ and $\tilde{X}$) on the order of finitely generated subgroups $\hat{G}$ of $G(\tilde{X},X)$, and thus $|G(\tilde{X},X)|\leq C_{\tilde{X}}/C_{X}$.
 \end{proof}
 
\subsubsection{Extension of Sormani's Theorem to $\RCD^{*}(0,N)$-spaces} 
Finally, we say that a length space $(X,\sfd)$ has the loop to infinity property if the following holds: for any element $g\in \bar{\pi}_1(X, x_0)$ and any compact subset $K\subset \subset X$, $g$ has a representative loop of the form 
$\sigma \circ \gamma \circ \sigma^{-1}$ where $\gamma$ is a loop in $X\setminus K$ and $\sigma$ is a curve from $x_0$ to $X\setminus K$.

We can then extend to $\RCD^*(0,N)$-spaces the loop to infinity property of manifolds with non-negative Ricci proved by Sormani \cite{So2}, and extended to Ricci limits in \cite{SW2}.

\begin{theorem}\label{thm:Sormani}
Let $(X,\sfd,\mm)$ be a non-compact $\RCD^*(0,N)$-space for some $N>1$. Then either $X$ has the loops at infinity property or the universal cover $(\tilde{X}, \sfd_{\tilde{X}},\mm_{\tilde{X}})$ splits isomorphically as metric measure space, i.e. it is isomorphic to a product  $(Y\times \R, \sfd_{Y\times \R}, \mm_{Y\times \R})$. 
\end{theorem}
 
 \begin{proof}
 We show that if $X$ does not have  the loops at infinity property then the universal cover $\tilde{X}$ must contain a line; since $(\tilde{X}, \sfd_{\tilde{X}},\mm_{\tilde{X}})$ is an $\RCD^{*}(0,N)$-space, the thesis will then follow by the Splitting Theorem \ref{thm:Split}. 
 \\
 
 \textbf{Step 1}. $X$ contains a ray, i.e.  an isometric immersion $\gamma:[0,\infty)\to X$; set  $x_{0}:=\gamma(0)$.
 
 Fix a reference point $x_{0}\in X$. Since by assumption $X$ is not compact, there exist $x_{k}\subset X$ such that $\sfd(x_{0},x_{n})\to \infty$. Let $\gamma_{n}:[0, \sfd(x_{0},x_{n})]$ be a length minimizing geodesic joining $x_{0}$ with $x_{n}$. Since $X$ is proper,  by using Ascoli-Arzel\'a Theorem we infer that for every $R>0$ we can find a sub-sequence of $n$'s such that $\gamma_{n}([0, \sfd(x_{0},x_{n})]) \cap B_{\tilde{x}_{0}}(R)$ converge uniformly to a length minimizing geodesic $\gamma^{R}:[0,R]\to X$ with $\gamma^{R}(0)=x_{0}$.  Considering now a sequence $R_{j}\to \infty$, via a diagonal argument, we finally obtain a ray  $\gamma:[0,\infty)\to X$ based at $x_{0}$.
 \\
 
  \textbf{Step 2}. We can find unit speed minimal geodesics $\tilde{\sigma}_{i}(t): [-r_{i},  r_{i}]\to \tilde{X}$, such that  $r_{i}\to +\infty$ and    $\tilde{\sigma}_{i}(0)$ are contained in a pre-compact subset of $\tilde{X}$.
  
  If $X$ does not have the loops at infinity property then there exists $1\neq g\in \bar{\pi}_1(X, x_0)$ and a compact subset $K\subset \subset X$ such that, called $C$ a loop based at $x_{0}$ representing $g$,  no closed geodesic contained in $X\setminus K$ can be homotopic to $C$ along the ray $\gamma$ constructed in Step 1.
\\  Let $R_{0}>0$ so that $K \subset B_{x_{0}}(R_{0})$, and let $r_{i}\geq R_{0}$ with $r_{i}\to +\infty$. By assumption, any loop based at $\gamma(r_{i})$ which is homotopic to $C$ along $\gamma$ must intersect $K$.
\\Let $(\tilde{X},\tilde{\sfd},\tilde{\mm})$ be the universal cover of $(X,\sfd,\mm)$ given by Theorem \ref{thm:UnCovRCD-noncompact}, and let $\tilde{C}$ be a lift of $C$ going from $\tilde{x}_{0}$ to $g \tilde{x}_{0}$. Since $g\neq 1$, clearly it holds  $\tilde{x}_{0} \neq g \tilde{x}_{0}$. Let $\tilde{\gamma}$ be the lift of $\gamma$ to $\tilde{X}$ starting at $\tilde{x}_{0}$ and let $g \tilde{\gamma}$ be the lift starting at $g \tilde{x}_{0}$.  Observe that if $\tilde{C}_{i}$ is a length minimizing geodesic from $\tilde{\gamma}(r_{i})$ to  $g\tilde{\gamma}(r_{i})$,  then the projection $C_{i}:=\pi(\tilde{C}_{i})$ is a loop based at $\gamma(r_{i})$ which is homotopic to $C$ along $\gamma$.  Thus there exists $t_{i}$ such that  $C_{i}(t_{i})\in K$.

Denote with  $L_{i}:=L_{\tilde{X}}(\tilde{C}_{i})=L_{X}(C_{i})=\tilde{\sfd}(\tilde{\gamma}(r_{i}), g \tilde{\gamma}(r_{i}))$ and let $\tilde{K}$ be the lift of $K$ to the fundamental domain of $X$ in $\tilde{X}$ such that $\tilde{x}_{0}\in \tilde{K}$. As $K$ is compact, clearly the lift $\tilde{K}$ is pre-compact. Since by construction $C_{i}(t_{i})\in K$, for every $i \in \N$ we can find $g_{i}\in G(\tilde{X},X)$ such that $g_{i}\tilde{C}_{i}(t_{i})\in \tilde{K}$. Observe that
$$
t_{i}=\tilde{\sfd}(\tilde{\gamma}(r_{i}), \tilde{C}_{i}(t_{i}))\geq \sfd(\gamma(r_{i}), C_{i}(t_{i}))\geq r_{i}-R_{0},
$$
and that
$$
L_{i}-t_{i}= \tilde{\sfd}(g \, \tilde{\gamma}(r_{i}), \tilde{C}_{i}(t_{i})) \geq   \sfd(\gamma(r_{i}), C_{i}(t_{i}))\geq r_{i}-R_{0}.
$$
Thus,  the curves $\tilde{\sigma}_{i}(t):=g_{i} \, \tilde{C}_{i} (t-t_{i})$ are unit speed minimal geodesics defined on $[-r_{i}+R_{0},  r_{i}-R_{0}]$  and such that  $\tilde{\sigma}_{i}(0)\in \tilde{K}$; since $r_{i}\to +\infty$, up to renaming $r_{i}$ with $r_{i}-R_{0}$, the claim follows.
\\

\textbf{Step 3}. Conclusion.

Since   $\tilde{\sigma}_{i}(0)$ are contained in a pre-compact subset of $\tilde{X}$, up to subsequences, we can assume that  $\tilde{\sigma}_{i}(0)\to \tilde{x}_{1}\in \tilde{X}$.  Since $\tilde{X}$ is proper, by using Ascoli-Arzel\'a Theorem we infer that for every $R>1$ we can find a sub-sequence of $i$'s such that $\tilde{\sigma}_{i}([-r_{i},  r_{i}]) \cap B_{R}(\tilde{x}_{1})$  converge uniformly to a length minimizing geodesic $\tilde{\sigma}^{R}:[-(R-1),(R-1)]\to \tilde{X}$ with $\tilde{\sigma}^{R}(0)=\tilde{x}_{1}$.  Considering now a sequence $R_{j}\to \infty$, via a diagonal argument, we finally obtain a line, i.e. an isometric immersion  $\tilde{\sigma}:\R \to \tilde{X}$.
  \end{proof}

\end{document}